\documentclass[11pt, leqno, oneside]{amsart} 

\usepackage{amsmath,amsthm,amssymb} 
\usepackage{amsfonts} 
\usepackage{txfonts} 
\usepackage{thmtools} 

\usepackage{geometry} 

\usepackage[dvipsnames, svgnames]{xcolor} 

\usepackage[shortlabels]{enumitem} 

\usepackage[colorlinks,citecolor=blue,hypertexnames=false,urlcolor=blue,breaklinks=true, pagebackref]{hyperref} 

\usepackage{accents} 
\usepackage{exscale} 
\usepackage{dsfont} 
\usepackage{mathrsfs} 
\usepackage[T1]{fontenc} 

\usepackage{etoolbox}
\appto\appendix{\addtocontents{toc}{\protect\setcounter{tocdepth}{1}}}

\renewcommand*\backref[1]{\ifx#1\relax \else (p. #1) \fi} 

\numberwithin{equation}{section} 

\theoremstyle{plain}
\newtheorem{theorem}[equation]{Theorem} 
\newtheorem{lemma}[equation]{Lemma}
\newtheorem{corollary}[equation]{Corollary}
\newtheorem{proposition}[equation]{Proposition}

\theoremstyle{definition}
\newtheorem{definition}[equation]{Definition}

\theoremstyle{remark}
\newtheorem{remark}[equation]{Remark}

\newtheorem{claim}[equation]{Claim}
\newtheorem{claim*}{Claim}

\newcommand{\bb}[1]{\mathbb{#1}}

\newcommand{\dist}{\operatorname{dist}}

\newcommand{\dv}{\operatorname{div}}
\newcommand{\Di}{\operatorname{D}}

\newcommand{\n}[1]{\mathscr{#1}}
\newcommand{\m}[1]{\mathcal{#1}}
\newcommand{\Lip}{\operatorname{Lip}} 
\newcommand{\ep}{\varepsilon}
\newcommand{\ra}{\rightarrow}
\newcommand{\lra}{\longrightarrow}   
\newcommand{\loc}{\operatorname{loc}}

\newcommand{\wt}{\widetilde}

\DeclareMathOperator{\supp}{supp}
\DeclareMathOperator{\diam}{diam}

\def\Xint#1{\mathchoice
	{\XXint\displaystyle\textstyle{#1}}%
	{\XXint\textstyle\scriptstyle{#1}}%
	{\XXint\scriptstyle\scriptscriptstyle{#1}}%
	{\XXint\scriptscriptstyle\scriptscriptstyle{#1}}%
	\!\int}
\def\XXint#1#2#3{{\setbox0=\hbox{$#1{#2#3}{\int}$}
		\vcenter{\hbox{$#2#3$}}\kern-.5\wd0}}

\def\dashint{\Xint-}

\def\YYint#1#2#3{{\setbox0=\hbox{$#1{#2#3}{\iint}$}
		\vcenter{\hbox{$#2#3$}}\kern-.51\wd0}}
 
\newcommand*{\dt}[1]{%
	\accentset{\mbox{\Large\bfseries .}}{#1}}

\allowdisplaybreaks

\setlist{nosep}

\begin{document}
	
\author[M. Mourgoglou]{Mihalis Mourgoglou}
\email[]{michail.mourgoglou@ehu.eus}
\address{Departamento de Matemáticas, Universidad del País Vasco, UPV/EHU, Barrio Sarriena S/N 48940 Leioa, Spain and, IKERBASQUE, Basque Foundation for Science, Bilbao, Spain}

\author[B. Poggi]{Bruno Poggi}
\email[]{brunopoggi@pitt.edu}
\address{Department of Mathematics, University of Pittsburgh, Pittsburgh, PA, 15213}

\title[Approximation result]{The Dirichlet problem as the boundary of the Poisson problem: A sharp approximation result}
\date{\today}

\thanks{M.M. was supported  by IKERBASQUE and partially supported by the grant PID2024-157724NB-I00: of the Ministerio de Econom\'ia y Competitividad (Spain), and by  IT-1615-22 (Basque Government). Part of this work was carried out while B.P. was supported by the European Research Council (ERC) under the European Union's Horizon 2020 research and innovation programme (grant agreement 101018680).
}
 
\begin{abstract} On a bounded   domain $\Omega\subset\bb R^{n+1}$, $n\geq2$, satisfying the corkscrew condition and with Ahlfors regular boundary, we  characterize the dual space to the space ${\bf N}_{2,p}$ of functions $u$ whose Kenig-Pipher modified non-tangential maximal operator $\m N_2(u)$ lies in $L^p(\partial\Omega)$, $p\in(1,\infty)$. We find that
\[
({\bf N}_{2,p})^*={\bf C}_{2,p'}\oplus L^{p'}(\partial\Omega),\qquad\text{and that}\qquad L^{p'}(\partial\Omega)=\partial^{\operatorname{weak}-*}{\bf C}_{2,p'}\,/\,{\bf C}_{2,p'},
\]
where ${\bf C}_{2,p'}$ is a certain $L^{p'}$-Carleson space and $p'$ is the H\"older conjugate of $p$. This  answers  a question considered by Hyt\"onen and Ros\'en. 

Inspired by this result and the recently understood characterizations of the $L^p$-solvability of the Dirichlet problem  in terms of the Poisson problem by Mourgoglou, Poggi, and Tolsa,  we show a novel approximation result: for an arbitrary elliptic operator $L=-\dv A\nabla$ with a  not necessarily symmetric  matrix $A$ of real bounded measurable coefficients, the solution space to the Dirichlet problem with data in $L^p(\partial\Omega)$
\[
	\left\{\begin{aligned}-\dv A\nabla u&=0,\quad&\text{in }&\Omega,\\u&=g,\quad&\text{on }&\partial\Omega,\end{aligned}\right.
\]
 lies on the weak-$*$ boundary in ${\bf N}_{2,p}$ of the solution space to the Poisson problem 
\[
\left\{\begin{aligned}-\dv A\nabla w&=-\dv F,\qquad&\text{in }&\Omega,\\ w&=0,\qquad&\text{on }&\partial\Omega,\end{aligned}\right. 
\]
with $F\in{\bf C}_{2,p}$, provided that the Dirichlet problem for $L$ with data in $L^p(\partial\Omega)$ is solvable in $\Omega$. This approximation result is sharp and new even for the Laplacian and on the unit ball.
\end{abstract} 

\maketitle

\hypersetup{linkcolor=blue}

	\tableofcontents

\section{Introduction}\label{sec.intro}

The $L^p$-solvability of homogeneous boundary value problems for second-order linear  elliptic partial differential equations   with non-smooth coefficients  on rough domains  has been an area of intense and fruitful research in the last few decades, culminating in shocking equivalences between certain geometric properties of domains   in the Euclidean space  and $L^p$-solvability of the homogeneous Dirichlet problem   for the Laplacian. In the recent article \cite{mpt25}, the authors discovered under no  assumptions on the real elliptic operator\footnote{The assumptions on the strongly elliptic operator in \cite{mpt25} for the stated result are as general as can be without degenerating the ellipticity or considering complex coefficients or systems.} and quite relaxed conditions on the geometry of the domain that the solvability of the Dirichlet problem with data in $L^p$
\begin{equation}\label{eq.dirichlet}
(\operatorname{D}^L_p)\qquad\left\{\begin{aligned}-\dv A\nabla u&=0,\quad&\text{in }&\Omega,\\u&=g,\quad&\text{on }&\partial\Omega,\\ \Vert\m N(u)\Vert_{L^p(\partial\Omega,\sigma)}&\lesssim\Vert g\Vert_{L^p(\partial\Omega,\sigma)},&&\end{aligned}\right.
\end{equation}
is equivalent to solvability of the Poisson-Dirichlet problem
\begin{equation}\label{eq.poisson}
(\operatorname{PD}^L_p)\qquad\left\{\begin{aligned}-\dv A\nabla w&=-\dv F,\qquad&\text{in }&\Omega,\\ w&=0,\qquad&\text{on }&\partial\Omega,\\  \Vert\widetilde{\m N}_2(w)\Vert_{L^p(\partial\Omega,\sigma)}&\lesssim\Vert\n C_2(F)\Vert_{L^p(\partial\Omega,\sigma)},&&\end{aligned}\right. 
\end{equation}
with inhomogeneous data $F$ in the  $L^p$-Carleson space ${\bf C}_{2,p}$. For the definitions of the non-tangential maximal operators $\m N$ and $\widetilde{\m N}_2$, as well as that of the Carleson function $\n C_2$ and the space ${\bf C}_{2,p}$, see Section \ref{sec.ntmax} and Definition \ref{def.spaces}. This quantitative equivalence of the problems (\ref{eq.dirichlet}) and (\ref{eq.poisson}) has garnered interest in the mathematical community and is an ongoing area of research across different settings and boundary value problems \cite{mz25,fl24,bmp25,fyy25}.

However, the  mechanism that binds the  boundary value problems (\ref{eq.dirichlet}) and (\ref{eq.poisson}) together is not well-understood at the granular level of the particular data. More precisely, in \cite{mpt25} it was shown that solvability of (\ref{eq.dirichlet}) for \emph{all} $g\in L^p(\partial\Omega)$ implies solvability of (\ref{eq.poisson}) for \emph{all} $F\in{\bf C}_{2,p}$, and viceversa. But now fix $g\in L^p(\partial\Omega)$, and the solution $u$ to (\ref{eq.dirichlet}) with this particular $g$. How do   solutions to the Poisson-Dirichlet problem (\ref{eq.poisson}) inform properties of this specific solution $u$? The proof of the equivalence  between $(\operatorname{D}^L_p)$ and $(\operatorname{PD}^L_p)$ in \cite{mpt25} used deep results about the harmonic measure,  Green's function, and their well-known relationship to each other and to the problem (\ref{eq.dirichlet}); however, in passing to these  kernels, one stops studying a particular solution $u$ and instead studies global properties of the solution operators to the respective problems.  At the end of this process, a robust quantitative equivalence between the \emph{global} problems $(\operatorname{D}^L_p)$ and $(\operatorname{PD}^L_p)$ is established \cite[Theorem 1.2 and Theorem 1.8]{mpt25}, but it remains unclear what is the relationship between the problems (\ref{eq.dirichlet}) and (\ref{eq.poisson}) at the level of the particular solutions. 


In this manuscript, we seek to gain insight at this deeper  level on the relationship between the problems $(\operatorname{D}^L_p)$ and $(\operatorname{PD}^L_p)$. Our first main result is the following novel approximation theorem. We defer all definitions to Section \ref{sec.prelim}.

\begin{theorem}[Approximation Theorem]\label{thm.approx} Let $\Omega\subset\bb R^{n+1}$, $n\geq2$, be a bounded   domain satisfying the corkscrew condition and with $n$-Ahlfors regular boundary, and let $A$ be a strongly elliptic, not necessarily symmetric real matrix of bounded measurable coefficients in $\Omega$. Fix $p\in(1,\infty)$, and suppose that $(\Di_p^L)$ is solvable in $\Omega$. Fix $g\in L^p(\partial\Omega)$, and let $u\in{\bf N}_{2,p}\cap W^{1,2}_{\loc}(\Omega)$ be the unique  solution to the   Dirichlet problem 
	\begin{equation}\label{eq.dirichlet0}
		\left\{\begin{aligned}-\dv A\nabla u&=0,\qquad&\text{in }&\Omega,\\ u&=g,\qquad&\text{on }&\partial\Omega.\end{aligned}\right. 
	\end{equation}
	Then  there exists a sequence of vector functions $\{F_j\}_{j\in\bb N}\subset\Lip_c(\Omega)^{n+1}$, with the following properties:
	\begin{enumerate}[(i)]
		\item\label{item.c1} There exists $C>0$ depending only on dimension, ellipticity, the corkscrew   constant, and the $n$-Ahlfors regularity constant, such that
		\begin{equation}\label{eq.uniformc}
			\Vert F_j\Vert_{{\bf C}_{2,p}}\leq C\Vert g\Vert_{L^p(\partial\Omega,\sigma)}.
		\end{equation}
		\item\label{item.c2} For each $j\in\bb N$, let $w_j$ be the unique weak solution in $W_0^{1,2}(\Omega)$ to the Poisson problem
		\begin{equation}\label{eq.poissondirichlet0}
			\left\{\begin{aligned}-\dv A\nabla w_j&=-\dv F_j,\qquad&\text{in }&\Omega,\\ w_j&=0,\qquad&\text{on }&\partial\Omega.\end{aligned}\right. 
		\end{equation}
		Then, $w_j\ra u$ as $j\ra\infty$ in the following topologies:
		\begin{enumerate}[(a)]
			\item\label{item.c2a} the strong topology in $C_{\loc}^{\alpha}(\Omega)$ for some $\alpha\in(0,1)$.
			\item\label{item.c2b} the strong topology in $W^{1,2}_{\loc}(\Omega)$.
			\item\label{item.c2c} the weak-$*$ topology in ${\bf N}_{2,p}$. Moreover, this convergence cannot be improved to convergence in the weak topology in ${\bf N}_{2,p}$ unless $g\equiv0$.
		\end{enumerate}
	\end{enumerate}
\end{theorem}

We prove this result in Section \ref{sec.approx}. Several remarks are in order. First, Theorem \ref{thm.approx} is completely new even in the classical setting $L=-\Delta$ and $\Omega=B(0,1)$. Second, certain technical parameters and the boundedness of the domain may be generalized under suitable modifications, see Remark \ref{rm.details}. Third, the sequence of approximating inhomogeneous data $\{F_j\}_j$ is constructed explicitly in terms of $g$ modulo regularizations (see (\ref{eq.fj0}) and (\ref{eq.fj})). 

The method of proof does not rely explicitly on properties of harmonic measure nor the Green's function. A fundamental object in the proof is the \emph{Varopoulos extension} of the boundary data $g$ \cite{var77,var78,hr18,ht21,mz25}. For its construction, we use the  state-of-the-art technology in \cite{mz25} (see Lemma \ref{lm.varopoulos}) which allows us to write Theorem \ref{thm.approx} in such high geometric generality. Since the Varopoulos extension is a nonlinear extension, we are careful with density arguments.

In Theorem \ref{thm.approx} \ref{item.c2}, the solutions $\{w_j\}_j$ to the Poisson problem are stated to converge to $u$ in three different  topologies. Of these, the most interesting one is in \ref{item.c2} \ref{item.c2c}, because convergence in the weak$-*$ topology of ${\bf N}_{2,p}$ is a global (i.e.\ up to the boundary of $\Omega$) convergence result, while the topologies in \ref{item.c2} \ref{item.c2a} and \ref{item.c2b} are fundamentally local in $\Omega$. Indeed, if one wanted to show only the local convergences in \ref{item.c2} \ref{item.c2a} and \ref{item.c2b}, this would be a simple exercise by setting $F_j=\nabla(u\zeta_j)$, where $\zeta_j$ is a cutoff function compactly supported in $\Omega\setminus B(\partial\Omega,1/j)$, and locally regularizing $u$ further if needed. We leave the details of this exercise to the reader. However, if one set up the $F_j$ in this way, then certainly neither \ref{item.c1} nor \ref{item.c2} \ref{item.c2c} would hold in general. 

Thus, the  contribution of Theorem \ref{thm.approx} is that it gives the correct (and heretofore unknown) \emph{global} approximation of solutions to the Dirichlet problem (\ref{eq.dirichlet0}) by solutions to the Poisson problem (\ref{eq.poissondirichlet0}) with trivial boundary data, as detailed in \ref{item.c2} \ref{item.c2c}, and with the appropriate uniform norm control implied by \ref{item.c1}. Theorem \ref{thm.approx} may appear counterintuitive at a first reading, because solutions to the Dirichlet problem are not trivial at the boundary, while solutions to the Poisson problem (\ref{eq.poissondirichlet0}) are always trivial at the boundary. This observation does preclude convergence in stronger topologies, and is essentially the reason that the convergence in \ref{item.c2} \ref{item.c2c} cannot be upgraded even to the weak topology in ${\bf N}_{2,p}$. 

We defer to Section \ref{sec.discuss} below several further remarks on Theorem \ref{thm.approx}.

Of course, to use the weak   and the weak-$*$ topologies of ${\bf N}_{2,p}$ in Theorem \ref{thm.approx} requires an understanding of the dual and predual spaces to ${\bf N}_{2,p}$. This program was initiated by Hyt\"onen and Ros\'en \cite{hr13} when $\Omega$ is the upper-half space.  Motivated by the  works \cite{aa11,ar12} on maximal regularity estimates and boundary value problems for nonsmooth elliptic systems, Hyt\"onen and Ros\'en asked in \cite[p.2]{hr13}  what the dual and predual of ${\bf N}_{2,p}$ are. In \cite{hr13} they showed that for all $p\in(1,\infty)$,
\begin{equation}\label{eq.hr13}\nonumber
{\bf C}_{2,p'}\subsetneq({\bf N}_{2,p})^*,\qquad\text{and}\qquad {\bf N}_{2,p}=({\bf C}_{2,p'})^*.
\end{equation} 
So the predual of ${\bf N}_{2,p}$ was identified, but its dual space was not.   Nevertheless, they proved   important quantitative estimates in \cite[Theorem 3.1]{hr13} (see also Proposition \ref{prop.duality} below) which showed that ${\bf C}_{2,p'}$ cannot miss too much of $({\bf N}_{2,p})^*$; more precisely,  the quotient space $({\bf N}_{2,p})^*/{\bf C}_{2,p'}$ is forced by \cite[Theorem 3.1]{hr13} to be in the   boundary  of ${\bf C}_{2,p'}$ in the weak-$*$ topology of $({\bf N}_{2,p})^*$; that is,
\begin{equation}\label{eq.weak*boundary}
	({\bf N}_{2,p})^*/{\bf C}_{2,p'}=(\partial^{\operatorname{weak}-*}{\bf C}_{2,p'})/{\bf C}_{2,p'}.
\end{equation}
These results were recently generalized to domains with significantly rougher geometry in \cite{mpt25}.

The second main result of this article is that we have managed to characterize the dual space to ${\bf N}_{2,p}$, answering the question posed by Hyt\"onen and Ros\'en in \cite{hr13}. For technical reasons, we use a   stricter definition of the space ${\bf N}_{2,p}$ than in \cite{hr13,mpt25}, whereby we consider the completion of $C(\overline\Omega)$ in the ${\bf N}_{2,p}$ norm; see Definition \ref{def.spaces} and Remark \ref{rm.different}. The use of this more restrictive definition is natural and goes back to the foundational work of Coifman, Meyer and Stein \cite{cms}, but see also Remark \ref{rm.different2}. Our precise result is as follows.

\begin{theorem}[The dual of ${\bf N}_{r,p}$]\label{thm.dual} Let $\Omega\subset\bb R^{n+1}$, $n\geq2$, be a    domain satisfying the corkscrew condition and with $n$-Ahlfors regular boundary, either bounded or with unbounded boundary, and let $r\in[1,\infty)$, $p\in(1,\infty)$, with $r'$,$p'$ the respective H\"older conjugates. Then the identity
	\begin{equation}\label{eq.dual}
		({\bf N}_{r,p})^*={\bf C}_{r',p'}\oplus L^{p'}(\partial\Omega,\sigma) 
	\end{equation}
	holds in the sense that
	\begin{enumerate}[(i)]
		\item For any $F\in{\bf C}_{r',p'}$ and $g\in L^{p'}(\partial\Omega,\sigma)$, the functional $T$ given by
		\[
		T(u):=\int_{\Omega}uF\,dm+\int_{\partial\Omega}ug\,d\sigma,\qquad u\in{\bf N}_{r,p},
		\]
		is a well-defined bounded linear functional on ${\bf N}_{r,p}$.
		\item If  $T\in({\bf N}_{r,p})^*$, then there exist unique $F\in{\bf C}_{r',p'}$ and $g\in L^{p'}(\partial\Omega,\sigma)$ such that 
		\begin{equation}\label{eq.decomp}
			T(u)=\int_{\Omega}uF\,dm+\int_{\partial\Omega}ug\,d\sigma,\qquad\text{for all }u\in{\bf N}_{r,p}.
		\end{equation}
	\end{enumerate}
\end{theorem}

See also Remark \ref{rm.hope} about the case $r=\infty$. Theorem \ref{thm.dual} is shown in Section \ref{sec.dual}. By Theorem \ref{thm.dual}, and (\ref{eq.weak*boundary}), we immediately conclude that
\begin{equation}\label{eq.weak*boundary2}
L^{p'}(\partial\Omega,\sigma)=(\partial^{\operatorname{weak}-*}{\bf C}_{r',p'})/{\bf C}_{r',p'}.
\end{equation}
The identity (\ref{eq.weak*boundary2}) is heavily related to the deep non-linear extension and trace results that Hyt\"onen and Ros\'en obtained in  \cite{hr18}. 

At this point we are ready to reveal a structural connection between Theorem \ref{thm.dual} and Theorem \ref{thm.approx}. The observant reader may have noticed that in the characterization (\ref{eq.dual}) of the dual of ${\bf N}_{2,p}$, the two subspaces ${\bf C}_{2,p'}$ and $L^{p'}(\partial\Omega,\sigma)$ correspond exactly to the spaces for the data in $(\operatorname{PD}_{p'}^L)$ and $(\operatorname{D}_{p'}^L)$ respectively. Furthermore, by (\ref{eq.weak*boundary2}), we know that the latter is the weak-$*$ boundary of the former. Does this identified structure in (\ref{eq.dual}), (\ref{eq.weak*boundary2}) ``lift'' to the level of  the solutions to the boundary value problems $(\operatorname{PD}_{p'}^L)$ and $(\operatorname{D}_{p'}^L)$? 

The answer is at least partially yes, and indeed, we first sought to answer the question of Hyt\"onen and Ros\'en about the dual of ${\bf N}_{2,p}$, and obtained Theorem \ref{thm.dual}. Having obtained (\ref{eq.dual}) and observed (\ref{eq.weak*boundary2}) as  its consequence, it is exactly the remarks and the question raised above that inspired us to seek and prove the approximation result, Theorem \ref{thm.approx}, which we now reinterpret as follows.

\begin{corollary}[$(\Di_p^L)$ lives on the boundary of $(\operatorname{PD}_p^L)$]\label{cor.meta} Retain the setting of Theorem \ref{thm.approx}. Let $\n D_p$ be the space of all solutions to the problem $(\Di_p^L)$ in (\ref{eq.dirichlet}), and let $\n P_p$ be the space of all solutions to the problem $(\operatorname{PD}_p^L)$ in (\ref{eq.poisson}). Then $\n D_p$ lies on the boundary of $\n P_p$ in the weak-$*$ topology of ${\bf N}_{2,p}$. That is,
\begin{equation}\label{eq.meta}
	\n D_p\subseteq(\partial^{weak-*}\n P_p)\,/\,\n P_p,
\end{equation}
as subspaces of ${\bf N}_{2,p}$.
\end{corollary}

Note the structural similarity between (\ref{eq.meta}) and (\ref{eq.weak*boundary2}). It is in the precise sense of Corollary \ref{cor.meta} that one can thus identify the Dirichlet problem with $L^p$ data as being a subset of ``the boundary'' of the Poisson-Dirichlet problem with inhomogeneous data in ${\bf C}_{2,p}$.

\subsection{Further remarks}\label{sec.discuss} 

It is natural to wonder whether an analogue of Theorem \ref{thm.approx} holds for   the classical Dirichlet problem (\ref{eq.dirichlet0}) with continuous boundary data, but this is not feasible. Indeed, in this case the uniform estimate which corresponds to the non-tangential maximal function estimate in (\ref{eq.dirichlet}) is   the classical  maximum principle,
\[
\Vert u\Vert_{C(\overline\Omega)}\lesssim\Vert g\Vert_{C(\partial\Omega)}.
\]
First, note that there is not even a well-defined weak-$*$ topology on $C(\overline\Omega)$  since this space does not admit a predual, because $\overline\Omega$ is an infinite compact metric space \cite[Chapter 6]{ddls16}. But the issue is not merely definitional: the uniform norm is too strong to control the natural candidates for  sequences of inhomogeneous data $\{F_j\}_j$ for the Poisson problem \ref{eq.poissondirichlet0}.  Thus it is not even clear that an analogue of the quantitative estimate (\ref{eq.uniformc}) is possible in this case.


A discussion about the approximation of solutions to the Dirichlet problem with $L^p$ data (\ref{eq.dirichlet}) would be remiss without a comparison to the relatively well-understood theory of \emph{$\ep$-approximators} \cite{KKoPT,HKMP,HMM2,gmt18,HT1,bptt24}. The $\ep$-approximators were often used to construct the Varopoulos extension of boundary data in $L^p(\partial\Omega)$, although recent techniques \cite{hr18,mz25} show that the existence of $\ep$-approximators to solutions of (\ref{eq.dirichlet}) is not required to construct Varopoulos extensions. Given $g\in L^p(\partial\Omega)$ and $u$ the solution to (\ref{eq.dirichlet}), there is a sequence $\{u_j\}$ of $\ep$-approximators which converge to $u$ in the \emph{strong} topology of ${\bf N}_{2,p}$, and $\nabla u_j\in{\bf C}_{2,p}$ for each $j\in\bb N$. However, in contrast to the sequence $\{w_j\}_j$ in Theorem \ref{thm.approx}, the $\ep$-approximators do not solve the Poisson-Dirichlet problem (\ref{eq.poisson}), as they are typically designed to stay close to $g$ on $\partial\Omega$, and thus far from trivial on $\partial\Omega$. Moreover, the sequence $\{\nabla u_j\}_j$ is not   in general bounded in ${\bf C}_{r,p}$ for any $r\in(1,\infty]$, whence it is not clear that there is a corresponding sequence of inhomogeneous data $\{F_j\}_j$ with $-\dv A\nabla u_j=-\dv F_j$ in $\Omega$ and such that the uniform estimate (\ref{eq.uniformc}) holds. We conclude that the approximators $\{w_j\}_j$ from Theorem \ref{thm.approx} are completely different from $\ep$-approximators in both qualitative and quantitative properties.

Finally, we emphasize that the principal challenge of this work lies in identifying the correct statements of our main results, Theorems \ref{thm.approx} and \ref{thm.dual}.   Once these formulations are in place, the overall strategy becomes clearer, although the proofs themselves remain technically demanding. 

We defer all further historical remarks to the survey in \cite{mpt25}.

\subsection{Outline} The rest of this manuscript is organized as follows. In Section \ref{sec.prelim} we introduce all relevant definitions and state useful technical lemmas. In Section \ref{sec.dual} we prove Theorem \ref{thm.dual}, and in Section \ref{sec.approx} we prove Theorem \ref{thm.approx}.

\section{Preliminaries}\label{sec.prelim}

If $x\in\bb R^{n+1}$ and $r>0$, by $B(x,r)$ we denote the Euclidean \emph{open} ball of radius $r$ centered at $x$. If $E\subset\bb R^{n+1}$ and $r>0$, we set $B(E,r):=\{x\in\bb R^{n+1}:\dist(x,E)<r\}$, and ${\bf 1}_E$ is the indicator function of $E$. We  write $a\lesssim b$ if there exists a constant $C>0$ so that $a\leq Cb$ and $a\lesssim_tb$ if $C$ depends on the parameter $t$. We write $a\approx b$ to mean $a\lesssim b\lesssim a$ and define $a\approx_tb$ similarly. Throughout, the letter $C$ denotes a constant that does not depend on the salient parameters in a given proof, and the value of $C$ might change from one line to the next.

We denote by $m=m_{d}$ the Lebesgue measure on $\bb R^{d}$, $d\in\bb N$. 

A set $E\subset\bb R^{n+1}$ is called \emph{$n$-Ahlfors regular} if there exists some constant $C_0>0$ such that
\[
C_0^{-1}r^n\leq\m H^n(B(x,r)\cap E)\leq C_0r^n,\qquad\text{for all }x\in E\text{ and all }0<r\leq\diam(E),
\]
where $\m H^n$ is the $n$-dimensional Hausdorff measure, which we assume to be normalized so that it coincides with $m_n$ in $\bb R^n$.

By a domain we mean an open set. If $\Omega$ is a domain in $\bb R^{n+1}$, we denote by $\sigma=\m H^n|_{\partial\Omega}$ the $n$-dimensional Hausdorff measure restricted to $\partial\Omega$. 

We say that a domain $\Omega$ satisfies the \emph{corkscrew condition} if there exists   $c>0$ such that for each $x\in\partial\Omega$ and every $r\in(0,2\diam\Omega)$, there exists a ball $B\subset B(x,r)\cap\Omega$ so that $r(B)\geq cr$.

Throughout the rest of this article, $\Omega$ denotes a domain in $\bb R^{n+1}$ satisfying the corkscrew condition and with $n$-Ahlfors regular boundary. 

Recall that $C_c^{\infty}(\Omega)$ is the space of compactly supported smooth functions in $\Omega$,  and  that for $p\in[1,\infty)$, $W^{1,p}(\Omega)$ is the Sobolev space of $p$-th integrable functions in $\Omega$ whose weak derivatives exist in $\Omega$ and are $p$-th integrable functions, while $W_0^{1,p}(\Omega)$ is the completion of $C_c^{\infty}(\Omega)$ under the norm $\Vert u\Vert_{W^{1,p}(\Omega)}:=\Vert u\Vert_{L^p(\Omega)}+\Vert\nabla u\Vert_{L^p(\Omega)}$. Moreover,   $\dt W^{1,p}(\Omega)$ consists of the $L^1_{\loc}(\Omega)$ functions whose weak gradient is $p$-th integrable over $\Omega$.


If $X$ is a metric space, by $\Lip_c(X)$ we denote the space of Lipschitz functions compactly supported in $X$.


\subsection{Non-tangential maximal functions and Carleson functions}\label{sec.ntmax}

For $\alpha>0$   and $\xi\in\partial\Omega$, we define the \emph{cone with vertex $\xi$ and aperture} $\alpha>0$ by
\begin{equation}\label{eq.cone}\nonumber
	\gamma_\alpha(\xi):=\{x\in\Omega:|x-\xi|<(1+\alpha)\delta(x)\}.
\end{equation}
Define the \emph{non-tangential maximal function} of  $u\in L^{\infty}_{\loc}(\Omega)$ by 
\begin{equation}\label{eq.ntmax}
\m N_\alpha(u)(\xi):=\sup_{x\in\gamma_\alpha(\xi)}|u(x)|,\quad\text{for }\xi\in\partial\Omega.
\end{equation}

Following \cite{kp}, we introduce the \emph{modified non-tangential maximal function} $\widetilde{\m N}_{\alpha,\hat c, r}$ for a given aperture $\alpha>0$, a parameter $\hat c\in(0,1/2]$, and $r\geq1$: for any $u\in L^r_{\loc}(\Omega)$, we write
\begin{equation}\label{eq.mnt}
	\widetilde{\m N}_{\alpha,\hat c,r}(u)(\xi):=\sup_{x\in\gamma_\alpha(\xi)}\Big(\dashint_{B(x,\hat c\delta(x))}|u(y)|^r\,dm(y)\Big)^{1/r},\qquad\xi\in\partial\Omega.
\end{equation}
For $p\in(1,\infty)$, the $L^p$ norms of non-tangential maximal functions are equivalent under changes of   $\alpha$ or the averaging parameter $\hat c$. For ease of notation, we will often write   $\widetilde{\m N}_{\alpha,r}=\widetilde{\m N}_{\alpha,\hat c,r}$ if we do not wish to specify $\hat c$. When we do not need to specify neither $\alpha$ nor $\hat c$,  we will write ${\m N}=\m N_\alpha$, $\gamma=\gamma_\alpha$, and $\widetilde{\m N}_r=\widetilde{\m N}_{\alpha,\hat c,r}$. 

Let $q\in[1,\infty]$ and $\hat c\in(0,1/2]$.   Define the \emph{$q$-Carleson functional} of a   function $H:\Omega\ra\bb R, H\in L_{\loc}^{q}(\Omega)$ by
\begin{equation}\label{eq.carlesonfn}
	\n C_{\hat c,q}(H)(\xi):=\sup_{r>0}\frac1{r^n}\int_{B(\xi,r)\cap\Omega}\Big(\dashint_{B(x,\hat c\delta(x))}|H|^{q}\Big)^{1/q}\,dm(x),\qquad \xi\in\partial\Omega.
\end{equation} 
The case $q=\infty$ is defined with the appropriate analogue. For $p\in(1,\infty)$, the $L^p$ norms of the Carleson functionals $\n C_{\hat c,q}$ defined in (\ref{eq.carlesonfn}) are equivalent under a change of the averaging parameter $\hat c$, and thus we    write $\n C_q=\n C_{\hat c,q}$ if we do not need to specify $\hat c$. 

\begin{definition}[The spaces ${\bf N}_{r,p}$ and ${\bf C}_{r,p}$]\label{def.spaces} Fix a   domain $\Omega$ satisfying the corkscrew condition and with $n$-Ahlfors regular boundary. By $C_c(\overline{\Omega})$ we denote the space of continuous functions on $\overline\Omega$ (the closure of $\Omega$) with compact support in $\overline\Omega$. Fix $\alpha_0>0$, $\hat c_0\in(0,\frac12]$, and write $\wt{\m N}_r:=\wt{\m N}_{\alpha_0,\hat c_0,r}$, and $\n C_r:=\n C_{\hat c_0,r}$. Note that the functionals
\[
\Vert u\Vert_{{\bf N}_{r,p}}:=\Vert\wt{\m N}_r(u)\Vert_{L^p(\partial\Omega,\sigma)} ,\qquad\qquad\Vert u\Vert_{{\bf C}_{r,p}}:=\Vert\n C_r(u)\Vert_{L^p(\partial\Omega,\sigma)}
\]
define norms on $C_c(\overline{\Omega})$ and on $C_c(\Omega)$, respectively.  Let ${\bf N}_{r,p}$ be the completion of $C_c(\overline\Omega)$ under the norm $\Vert\cdot\Vert_{{\bf N}_{r,p}}$, and let ${\bf C}_{r,p}$ be the completion of $C_c(\Omega)$ under the norm   $\Vert\cdot\Vert_{{\bf C}_{r,p}}$. By \cite[Lemma 2.5]{mpt25}, if $r,p\in(1,\infty)$, we may also identify ${\bf C}_{r,p}$ as the Banach space of functions $u\in L^r_{\loc}(\Omega)$ such that $\Vert\n C_r(u)\Vert_{L^p(\partial\Omega)}<+\infty$. 
\end{definition}

The following proposition establishes a certain quantitative quasi-duality between the spaces ${\bf N}_{r,p}$ and ${\bf C}_{r',p'}$.   When $\Omega$ is the half-space, these estimates were shown by Hyt\"onen and Ros\'en in \cite[Theorem 3.1, Theorem 3.2]{hr13}. The general case cited here was shown in   \cite{mpt25}.

\begin{proposition}\label{prop.duality} Let $\Omega\subset\bb R^{n+1}$, $n\geq1$, be a domain satisfying the corkscrew condition and such that $\partial\Omega$ is $n$-Ahlfors regular. Suppose that either $\Omega$ is bounded, or that $\partial\Omega$ is unbounded. Let $p,q\in(1,\infty)$ and $p'$, $q'$ their H\"older conjugates. Then ${\bf N}_{q,p}=({\bf C}_{q',p'})^*$, and moreover,
	\begin{equation}\label{eq.hr7}\nonumber
		\Vert uH\Vert_{L^1(\Omega)}\lesssim\Vert\wt{\m N}_q(u)\Vert_{L^p(\partial\Omega)}\Vert\n C_{q'}H\Vert_{L^{p'}(\partial\Omega)},\qquad u\in L^q_{\loc}(\Omega), H\in L^{q'}_{\loc}(\Omega),
	\end{equation}
	\begin{equation}\label{eq.hr8}\nonumber
		\Vert\widetilde{\m N}_q(u)\Vert_{L^p(\partial\Omega)}\lesssim\sup_{H:\Vert\n C_{q'}(H)\Vert_{L^{p'}(\partial\Omega)}=1}\Big|\int_{\Omega}Hu\,dm\Big|,\qquad u\in L^q_{\loc}(\Omega),
	\end{equation}
	\begin{equation}\label{eq.hr9}\nonumber
		\Vert\n C_{q'}H\Vert_{L^{p'}(\partial\Omega)}\lesssim\sup_{u:\Vert\wt{\m N}_qu\Vert_{L^p(\partial\Omega)}=1}\Big|\int_\Omega Hu\,dm\Big|,\qquad H\in L^{q'}_{\loc}(\Omega).
	\end{equation}
\end{proposition}

\begin{remark}\label{rm.different} Observe that our definition of ${\bf N}_{r,p}$ in this manuscript technically differs from the definition of the same symbol in \cite{mpt25}, or from the analogous symbol in \cite{hr13}. For each $r\in[1,\infty)$ and $p\in(1,\infty)$, let $\hat{\bf N}_{r,p}$ be the Banach space of functions $u\in L^r_{\loc}(\Omega)$ such that $\Vert\wt{\m N}_r(u)\Vert_{L^p(\partial\Omega)}<+\infty$. In \cite{hr13,mpt25}, the space $\hat{\bf N}_{r,p}$ was considered, rather than the the completion of $C_c(\overline{\Omega})$ under the norm $\Vert\cdot\Vert_{{\bf N}_{r,p}}$. Notice that ${\bf N}_{r,p}\subset\hat{\bf N}_{r,p}$ trivially. The essential distinction between these spaces, from our point of view, is that arbitrary elements in $\hat{\bf N}_{r,p}$ need not have well-defined traces on $\partial\Omega$ (since they may oscillate uncontrollably near $\partial\Omega$ while still maintaining finite $\Vert\cdot\Vert_{{\bf N}_{r,p}}$ norm), but arbitrary elements in ${\bf N}_{r,p}$ always have well-defined traces on $\partial\Omega$, as we will see in short order. This subtle change in the definition of ${\bf N}_{r,p}$ will allow for a significantly cleaner characterization of its dual space in Theorem \ref{thm.dual}. On the other hand, the heart of the matter is not sacrificed, since our techniques also yield extensive insight into the arbitrary elements in the dual space of $\hat{\bf N}_{r,p}$, as explained in Remark \ref{rm.different2}. 
\end{remark}

\begin{lemma}\label{lm.trace} Let $\Omega\subset\bb R^{n+1}$, $n\geq2$, be a   domain satisfying the corkscrew condition and with $n$-Ahlfors regular boundary, let $r\in[1,\infty)$, $p\in(1,\infty)$ and $u\in{\bf N}_{r,p}$. Then $u|_{\partial\Omega}\in L^p(\partial\Omega,\sigma)$ (in the sense of traces).
\end{lemma}

\noindent\emph{Proof.} Fix $u\in{\bf N}_{r,p}$. Then there exists a sequence $\{u_k\}\subset C_c(\overline\Omega)$ such that $u_k\ra u$ strongly in ${\bf N}_{r,p}$. Since
\[
v(\xi)\leq\wt{\m N}_r(v)(\xi),\qquad\text{for all }\xi\in\partial\Omega, \text{ and all } v\in C_c(\overline\Omega),
\]
then note that for all $k,\ell\in\bb N$,
\[
\Vert u_k-u_{\ell}\Vert_{L^p(\partial\Omega,\sigma)}\leq\Vert u_k-u_{\ell}\Vert_{{\bf N}_{r,p}}\lra0\quad\text{ as }k,\ell\ra\infty,
\]
so that $\{u_k\}$ is Cauchy in $L^p(\partial\Omega)$. Let $g_u:=\lim_ku_k$ in the strong $L^p$ topology. It is easy to check that $g_u$, as an element in $L^p(\partial\Omega)$, is independent of the approximating sequence $\{u_k\}$.  Since $u$ was arbitrary in ${\bf N}_{r,p}$, then by a standard argument, there is a bounded linear trace operator $T:{\bf N}_{r,p}\ra L^p(\partial\Omega,\sigma)$ given by $T(u)=g_u$ for all $u\in{\bf N}_{r,p}$, which satisfies that
\begin{equation}\label{eq.trace}
\Vert g_u\Vert_{L^p(\partial\Omega)}\leq\Vert u\Vert_{{\bf N}_{r,p}}.
\end{equation}
It is clear that $u=g_u$ $\sigma$-a.e.\ on $\partial\Omega$ whenever $u\in C_c(\overline\Omega)$. We thus write $u|_{\partial\Omega}:=g_u$, and the desired result follows.\hfill{$\square$}

\subsection{Elliptic PDE preliminaries}\label{sec.pde}
We assume throughout that $A$ is a real, not necessarily symmetric $(n+1)\times(n+1)$ matrix of   merely  bounded measurable coefficients in $\Omega$ verifying the strong ellipticity conditions 
\begin{equation}\label{eq.elliptic}
	\lambda|\xi|^2\leq\sum_{i,j=1}^{n+1}A_{ij}(x)\xi_i\xi_j,\qquad \Vert A\Vert_{L^{\infty}(\Omega)}\leq\frac1{\lambda},\quad x\in\Omega,\quad\xi\in\bb R^{n+1}.
\end{equation}  
Define the elliptic operator $L$ acting formally on real-valued functions $u$ by
\[
Lu=-\dv(A\nabla u)=-\sum_{i,j=1}^{n+1}\frac{\partial}{\partial x_i}\Big(a_{ij}\frac{\partial u}{\partial x_j}\Big).
\]
We write $A^T$ for the transpose of $A$, and $L^*=-\dv A^T\nabla$.

We will also make use of the conormal derivative. Let $\m T(\partial\Omega)$ denote the trace space of $W^{1,2}(\Omega)$; that is, $\m T(\partial\Omega)$ is the quotient space $W^{1,2}(\Omega)/W^{1,2}_0(\Omega)$.

\begin{definition}[Conormal Derivative]\label{def.conormal} Given $H\in\Lip_c(\Omega)$ and the unique weak solution $v\in W_0^{1,2}(\Omega)$ to the Poisson  problem
	\begin{equation}\label{eq.poissonh}
		\left\{\begin{aligned}-\dv A\nabla v&=H,\qquad&\text{in }&\Omega,\\ v&=0,\qquad&\text{on }&\partial\Omega.\end{aligned}\right. 
	\end{equation}
	define the functional
	\begin{equation}\label{eq.funct}
		\ell_v(\varphi)=B[v,\Phi]:=\int_\Omega A\nabla v\nabla\Phi\,dm-\int_\Omega H\Phi\,dm,\qquad\varphi\in\m T(\partial\Omega),
	\end{equation}
	where $\Phi\in W^{1,2}(\Omega)$ satisfies $\operatorname{Tr}\Phi=\varphi$.  It is easy to see that $\ell_v(\varphi)$ is well-defined and independent of the extension $\Phi$, since $v$ solves the equation $Lv=H$ in the weak sense. We call $\ell_v(\varphi)$ the \emph{conormal derivative of $v$}, and denote $\partial_{\nu_A}v=\ell_v$.
\end{definition}

\begin{definition}[The Dirichlet problem $(\operatorname{D}_p^L)$]\label{def.dirichlet} Let $p\in(1,\infty)$. We say that the (homogeneous) \emph{Dirichlet problem for the operator $L$ with $L^{p}$ data in $\Omega$ is solvable} (write $(\Di_{p}^L)$ is solvable in $\Omega$), if there exists $C\geq1$ so that for each $g\in C(\partial\Omega)$, the solution $u$ to the continuous Dirichlet problem 
\begin{equation}\label{eq.dirichlet2}\nonumber
\left\{\begin{aligned}-\dv A\nabla u&=0,\quad&\text{in }&\Omega,\\u&=g,\quad&\text{on }&\partial\Omega,\end{aligned}\right.
\end{equation}
satisfies the estimate 
\begin{equation}\label{eq.direst}\nonumber
\Vert\m N(u)\Vert_{L^{p}(\partial\Omega,\sigma)}\leq C\,\Vert g\Vert_{L^{p}(\partial\Omega,\sigma)}.
\end{equation}
\end{definition}

\subsection{Useful lemmas}

The following extension result is fundamental to our program.

\begin{lemma}[The Varopoulos extension, {\cite[Theorem 1.4]{mz25}}]\label{lm.varopoulos} Let $\Omega\subset\bb R^{n+1}$, $n\geq2$, be a    domain satisfying the corkscrew condition and with $n$-Ahlfors regular boundary, and fix $g\in \Lip_c(\partial\Omega)$. Then there exists a function $\Phi:\overline{\Omega}\ra\bb R$ with the following properties:
	\begin{enumerate}[(i)]
		\item\label{item.v1} $\Phi\in C^{\infty}(\Omega)\cap\Lip(\overline{\Omega})\cap\dt W^{1,2}(\Omega)$.
		\item\label{item.v2} $\Vert\m N(\Phi)\Vert_{L^p(\sigma)}+\Vert\n C_\infty(\nabla\Phi)\Vert_{L^p(\sigma)}\lesssim\Vert g\Vert_{L^p(\sigma)}$, for each $p\in(1,\infty]$.
		\item\label{item.v3} $\Vert\m N(\delta_\Omega\nabla\Phi)\Vert_{L^p(\sigma)}\lesssim\Vert g\Vert_{L^p(\sigma)}$.
		\item\label{item.v4} $\Phi|_{\partial\Omega}=g$ continuously.
	\end{enumerate}
	We call $\Phi$ the \emph{Varopoulos extension} of $g$ onto $\Omega$.
\end{lemma}

We will also make use of the following technical approximation of integrable functions on the boundary of a domain by continuous functions on the boundary, essentially mentioned in  \cite[Lemma 3.5]{chm} without proof. The result is indeed a straightforward generalization of the classical theory of convolutions and approximate identities, thus its proof is omitted.

\begin{lemma}\label{lm.approxid} Let $\varphi\in C_c^\infty(\bb R)$ be such that ${\bf 1}_{(0,1)}\leq\varphi\leq{\bf 1}_{(0,2)}$. For each $t>0$ and  $h\in L^1_{\loc}(\partial\Omega,\sigma)$ define
\[
P_th(\xi):=\int_{\partial\Omega}\varphi_t(\xi,\zeta)h(\zeta)\,d\sigma(\zeta),\qquad\xi\in\partial\Omega,
\]
where
\[
\varphi_t(\xi,\zeta):=\frac1{\int_{\partial\Omega}\varphi\big(\frac{|\xi-z|}t\big)\,d\sigma(z)}\varphi\Big(\frac{|\xi-\zeta|}t\Big),\qquad\xi,\zeta\in\partial\Omega.
\]
The following hold:
\begin{enumerate}[(i)]
	\item If $h\in L^q(\partial\Omega,\sigma)$, $q\in[1,\infty)$, then $P_th\in L^\infty(\partial\Omega)\cap C(\partial\Omega)$, and $P_th\lra h$ strongly in $L^q(\partial\Omega,\sigma)$ as $t\ra0$.
	\item If $h$ is continuous at $\xi\in\partial\Omega$ (as a function on $\partial\Omega$), then $P_th(\xi)$ converges to $h(\xi)$ as $t\ra0$.
	\item If $h\in L^q(\partial\Omega,\sigma)$, $q\in[1,\infty]$, with $\supp h$ is contained in a ball of radius $r$, then for each $t>0$, $\supp P_th$ is contained in the concentric ball of radius $r+2t$.
\end{enumerate}
\end{lemma}

Next we record the easily shown fact that for each compact set $K\subset\Omega$, the space $L^r(K)$ embeds continuously into ${\bf N}_{r,p}$.

\begin{lemma}\label{lm.local} Let $\Omega\subset\bb R^{n+1}$, $n\geq2$, be a  domain satisfying the corkscrew condition and with $n$-Ahlfors regular boundary, let $K$ be a compact subset of $\Omega$, and $r\in[1,\infty)$, $p\in(1,\infty)$. Then there exists a constant $C_K$   such that for every $u\in L^r(K)$ with $u\equiv0$ on $\overline\Omega\setminus K$, we have the estimate
	\[
	\Vert\wt{\m N}_r(u)\Vert_{L^p(\partial\Omega)}\leq C_K\Vert u\Vert_{L^r(K)}.
	\]
\end{lemma}

The following technical result shows that ${\bf N}_{r,p}$ misses from $\hat{\bf N}_{r,p}$  only the behavior at the boundary of the domain (see Remark \ref{rm.different} for the definition of $\hat{\bf N}_{r,p}$); it will be useful later in the proof of Theorem \ref{thm.dual}. Its omitted proof is a  straightforward density argument using Lemma \ref{lm.local}.

\begin{lemma}\label{lm.nc}  Let $\Omega\subset\bb R^{n+1}$, $n\geq2$, be a  domain satisfying the corkscrew condition and with $n$-Ahlfors regular boundary, let $K$ be a compact subset of $\Omega$, $r\in[1,\infty)$, $p\in(1,\infty)$ and $u\in L^r_{\loc}(\Omega)$ with $\Vert\wt{\m N}_r(u)\Vert_{L^p(\partial\Omega)}<+\infty$. Then $u{\bf1}_K\in{\bf N}_{r,p}$. More precisely, there is a sequence $\{v_j\}_j\subset C_c(K)$ such that $v_j\ra u{\bf 1}_K$ strongly in ${\bf N}_{r,p}$ and in $L^r(K)$ as $j\ra\infty$.
\end{lemma}

\section{The dual of the space ${\bf N}_{r,p}$}\label{sec.dual}

In this section, we prove Theorem \ref{thm.dual}.

\begin{remark}\label{rm.hope} There is no hope for the identity (\ref{eq.dual}) to be valid with $r=+\infty$. Indeed, note that ${\bf N}_{\infty,p}$ coincides exactly with the completion of $C_c(\overline{\Omega})$ in the norm $\Vert\m N(u)\Vert_p$, where $\m N$ is the classical non-tangential maximal function (with a possibly slightly larger cone aperture), and the dual space to this space is well-understood. For instance, if $\Omega$ is the half-space, then the dual space of ${\bf N}_{\infty,1}$ is known to be precisely the space of Carleson measures \cite{cms}, and Carleson measures  need not obey the decomposition in (\ref{eq.dual}). As a concrete and easy example, consider the Dirac delta measure with pole at a point $x\in\Omega$; then it is straightforward to explicitly compute that this is a Carleson measure and that it represents a bounded linear functional on ${\bf N}_{\infty,p}$ for any $p\in(1,\infty)$.
\end{remark}

\noindent\emph{Proof of Theorem \ref{thm.dual}.} Fix $r\in[1,\infty)$, $p\in(1,\infty)$. From \cite{hr13,mpt25}, it is already known that
\begin{equation}\label{eq.dual0}
{\bf C}_{r',p'}\subset({\bf N}_{r,p})^*.
\end{equation}
We show that
\begin{equation}\label{eq.dual1}
L^{p'}(\partial\Omega,\sigma)\subset ({\bf N}_{r,p})^*,
\end{equation}
and that 
\begin{equation}\label{eq.dual2}
({\bf N}_{r,p})^*\subset{\bf C}_{r',p'}\oplus L^{p'}(\partial\Omega,\sigma),
\end{equation}
with the latter containment being the challenging one.  Note that (\ref{eq.dual0}), (\ref{eq.dual1}), and (\ref{eq.dual2}) together   imply the desired identity (\ref{eq.dual}). In particular, (\ref{eq.dual1}) and (\ref{eq.dual2}) together imply the statement (i) of the theorem.

{\bf Proof of }(\ref{eq.dual1}). Fix $g\in L^{p'}(\partial\Omega)$.   Define the functional
\[
T_g(u):=\int_{\partial\Omega}u|_{\partial\Omega}\,g\,d\sigma,\qquad u\in{\bf N}_{r,p},
\]
and note that, by Lemma \ref{lm.trace}, $T_g$ is well-defined on ${\bf N}_{r,p}$. Moreover, by H\"older's inequality,
\[
|T_g(u)|\leq\Vert u\Vert_{L^p(\partial\Omega)}\Vert g\Vert_{L^{p'}(\partial\Omega)}\leq\Vert u\Vert_{{\bf N}_{r,p}}\Vert g\Vert_{L^{p'}(\partial\Omega)},
\]
where we used (\ref{eq.trace}). Since $g$ was arbitrary, the containment in (\ref{eq.dual1}) is shown.

{\bf Proof of }(\ref{eq.dual2}). Let $T$ be a bounded   linear functional on ${\bf N}_{r,p}$. In particular, $T$ is a bounded   linear functional on $C_c(\overline\Omega)$. By the Riesz Representation Theorem, there exists a unique (complex or signed) regular Borel measure $\mu$ on $\overline\Omega$ such that
\begin{equation}\label{eq.ide}
T(u)=\int_{\overline\Omega}u\,d\mu,\qquad\text{for each }u\in C_c(\overline\Omega).
\end{equation}
Consider the following claim.

\begin{claim}\label{claim.decomp} The measure $\mu$ decomposes as
\begin{equation}\label{eq.omega}
d\mu = F\,dm + g\,d\sigma,
\end{equation}
where $m$ is the $(n+1)$-dimensional Lebesgue measure restricted to $\Omega$, $\sigma$ is the $n$-dimensional Hausdorff measure on $\partial\Omega$, $F\in{\bf C}_{r',p'}$, and $g\in L^{p'}(\partial\Omega,\sigma)$.
\end{claim}

Assume Claim \ref{claim.decomp} for the moment. Let $\wt T$ be the functional given by
\[
\wt T(u):=\int_{\Omega}uF\,dm+\int_{\partial\Omega}ug\,d\sigma,\qquad u\in{\bf N}_{r,p}.
\]
By Carleson's Theorem and H\"older's inequality, it is clear that $\wt T$ is a bounded linear functional on ${\bf N}_{r,p}$. Moreover, Claim \ref{claim.decomp} gives that $T(u)=\wt T(u)$ for all $u\in C_c(\overline\Omega)$. However, if $u\in{\bf N}_{r,p}$, then there exists a sequence $\{u_k\}\subset C_c(\overline\Omega)$ with $u_k\ra u$ in ${\bf N}_{r,p}$, and
\[
T(u)=\lim_kT(u_k)=\lim_k\wt T(u_k)=\wt T(u).
\]
Thus the identity (\ref{eq.decomp}) holds. Finally, $F$ and $g$ are unique since $\mu$ is unique by the Riesz Representation Theorem. This ends the proof of (ii) of the theorem, modulo the proof of Claim \ref{claim.decomp}.

\noindent\emph{Proof of Claim \ref{claim.decomp}.} By the linearity of the integral and the linearity of $T$, without loss of generality we treat only the case where $\mu$ is a positive measure on $\overline\Omega$. In this case, the identity (\ref{eq.ide}) gives us that $T$ is a positive linear functional on $C_c(\overline\Omega)$. Define
\[
\Omega_\tau:=\big\{x\in \Omega: \delta(x)> \tau\big\},\qquad\text{for each }\tau>0.
\]
We break the proof of the claim into several parts.

{\bf Absolute continuity of $\mu$ with respect to $m$ in compact subsets of $\Omega$.} Fix a compact set $K\subset\Omega$ and let $E\subset K$ be an arbitrary Borel set. We will prove that
\begin{equation}\label{eq.abscont}
\mu(E)\leq C m(E)^{\frac1r},
\end{equation}
Since $K\subset\Omega$ is compact and $\Omega$ is an open set,  $k:=\dist(K,\partial\Omega)>0$, and $\ell:=\diam(K)<+\infty$. Fix $z\in E$, let $U\subset\Omega$ be an arbitrary  open set that contains $E$, and for each $\tau>0$, define
\[
U_\tau:=U\cap\Omega_\tau\cap B(z,1000\ell).
\]
Note that for all $\tau,t\in(0,k]$, $U_\tau$ is an open set, $E\subseteq U_\tau\subseteq U_t\subseteq U$ if $t<\tau$, and   $\dist(U_\tau,\partial\Omega)\geq\tau$. Now let $\phi\in C_c(U_{k/2})$ be a continuous bump function with $\phi\equiv1$ on $U_k$ and $0\leq\phi\leq 1$. Then we have that
\[
\mu(U_k)=\int_{U_k}d\mu=\int_{\overline\Omega}{\bf 1}_{U_k}\,d\mu\leq\int_{\overline\Omega}\phi\,d\mu=T(\phi)\leq C\Vert\wt{\m N}_r(\phi)\Vert_{L^p(\partial\Omega)}.
\]
Let us control the term on the right-hand side. Note that for any $x\in\Omega$, if there exists $y\in B(x,\delta(x)/2)\cap U_{k/2}$, then $k/3<\delta(x)<C\ell$. Since $\supp\phi\in U_{k/2}$, it follows that
\[
\supp\wt{\m N}_r(\phi)\subset B(z,C\ell).
\]
Moreover, for each $\xi\in\partial\Omega$ and each $x\in\gamma(\xi)$, we have that 
\begin{equation}\nonumber
 \Big(\dashint_{B(x,\frac{\delta(x)}2)}|\phi|^r\,dm\Big)^{\frac1r}\leq\Big(\dashint_{B(x,\frac{\delta(x)}2)}{\bf 1}_{U_{k/2}}\,dm\Big)^{\frac1r}\leq C\delta(x)^{-\frac{n+1}r}m(U_{k/2})^{\frac1r}  \leq Ck^{-\frac{n+1}r}m(U_{k/2})^{\frac1r},
\end{equation}
where $C$ depends only on $n$ and $r$. Putting these observations together, it follows that
\[
\mu(U_k)\leq C\Vert\wt{\m N}_r(\phi)\Vert_{L^p(\partial\Omega)}\leq C\sigma\big(B(z,C\ell)\cap\partial\Omega\big)^{\frac1p}k^{-\frac{n+1}r}m(U_{k/2})^{\frac1r},
\]
or succinctly,
\[
\mu(U_k)\leq   C m(U_{k/2})^{\frac1r},\qquad\text{for all open }U\supset E.
\]
By the outer regularity of $\mu$, we see that
\[
\mu(E)\leq\mu(U_k),\qquad\text{for all open }U\supset E,
\]
so that
\begin{equation}\label{eq.abscont0}
\mu(E)\leq C m(U_{k/2})^{\frac1r},\qquad\text{for all open }U\supset E.
\end{equation}
Taking infimum over all open $U$ in $\Omega$ containing $E$ and using the outer regularity of the Lebesgue measure $m$, from (\ref{eq.abscont0}) we get (\ref{eq.abscont}), as desired.

Consequently, $\mu\ll m$ on compact sets in  $\Omega$. By a standard diagonalization argument it follows   that there exists a non-negative, locally integrable Borel measurable function $F:\Omega\ra\bb R$ such that $d\mu|_K=F\,dm|_K$ for any compact set $K\subset\Omega$; and moreover, for any $u\in C_c(\overline\Omega)$,
\begin{equation}\label{eq.decompose1}
	\int_{\overline\Omega}u\,d\mu=\int_{\Omega}u\,d\mu+\int_{\partial\Omega}u\,d\mu=\int_{\Omega}uF\,dm+\int_{\partial\Omega}u\,d\mu.
\end{equation}

{\bf Integrability property of $F$.} Let us show now that $F\in L^{r'}_{\loc}(\Omega)$. Fix a compact subset $K$ of $\Omega$. By the density of continuous functions in $L^{r}(K)$ and the duality of Lebesgue spaces (recall that $r\in[1,\infty)$), we know that $F\in L^{r'}(K)$ if and only if there exists a constant $C$ such that for each $v\in C_c(K)$, the estimate
\begin{equation}\label{eq.localint}\nonumber
\Big|\int_KvF\,dm\Big|\leq C\Vert v\Vert_{L^r(K)}
\end{equation}
holds. But any arbitrary $v\in C_c(K)$ can be extended trivially to $C_c(\overline\Omega)$ with $v\equiv0$ on $\overline\Omega\setminus K$. Using (\ref{eq.decompose1}) and (\ref{eq.ide}), it follows that
\[
\Big|\int_KvF\,dm\Big|=\Big|\int_\Omega vF\,dm\Big|=\Big|\int_{\overline\Omega}v\,d\mu\Big|=|T(v)|\leq C\Vert\wt{\m N}_rv\Vert_p\lesssim\Vert v\Vert_{L^r(K)},
\]
where in the last estimate we used Lemma \ref{lm.local}. We conclude that $F\in L^{r'}(K)$, as desired.

{\bf Carleson property of $F$.} We now show that $F\in{\bf C}_{r',p'}$. Fix $z\in\Omega$ such that $F(z)\neq0$,  for each $k\in\bb N$, let $K_k:=\Omega_{1/k}\cap B(z,1/k)$, and let $F_k:\Omega\ra\bb R$ be the function defined as
\[
F_k(x):=F(x){\bf 1}_{K_k}(x),\qquad\text{for each }x\in\Omega.
\]
It is clear that for each $k$, $F_k$ is a non-negative Borel-measurable, integrable function on $\Omega$, with  support in the compact set $K_k$, and  with $F_k\leq F_{k+1}$ on $\Omega$. Moreover we have that $F_k\ra F$ pointwise as $k\ra\infty$ in $\Omega$. 

We first show that for each $k$, $F_k\in{\bf C}_{r',p'}$.   By \cite[Proposition 2.3]{mpt25}, we have that
\begin{equation}\label{eq.est}
\Vert F_k\Vert_{{\bf C}_{r',p'}}\lesssim\sup_{u\in L^r_{\loc}(\Omega)\,:\,\Vert\wt{\m N}_ru\Vert_p=1}\Big|\int_\Omega uF_k\,dm\Big|,
\end{equation}
where we know the integral on the right-hand side is finite (for each fixed $u$) since $F_k\in L^{r'}(K_k)$ and $F_k$ is supported on $K_k$. Fix arbitrary $u\in L^r_{\loc}(\Omega)$ with $\Vert\wt{\m N}_ru\Vert_p=1$, and note  that
\[
\int_\Omega uF_k\,dm=\int_{\Omega}\Big(u(x){\bf 1}_{K_k}(x)\Big)F(x)\,dm(x).
\]
By Lemma \ref{lm.nc},   there is a sequence $\{v_j\}_j\subset C_c(K_k)$ such that $v_j\ra u{\bf1}_{K_k}$ strongly in ${\bf N}_{r,p}$ and in $L^r(K)$ as $j\ra\infty$, which coupled with the above identity implies that
\[
\int_\Omega uF_k\,dm=\lim_{j\ra\infty}\int_Kv_jF\,dm,
\]
since $F\in L^{r'}(K)$. On the other hand, using (\ref{eq.ide}), (\ref{eq.decompose1}), and the fact that $T$ is a bounded linear functional on ${\bf N}_{r,p}$, for each $j\in\bb N$ we have that
\[
\Big|\int_\Omega v_jF\,dm\Big|=|T(v_j)|\leq C\Vert v_j\Vert_{{\bf N}_{r,p}},
\]
where $C$ is independent of $j$ and $k$. When passing $j\ra\infty$, we see that
\[
\Big|\int_\Omega uF_k\,dm\Big|\leq C,
\]
Since $u$ was arbitrary, from (\ref{eq.est}) we conclude that 
\begin{equation}\label{eq.fk}
\Vert F_k\Vert_{{\bf C}_{r',p'}}\leq C.
\end{equation}
Hence $F_k\in{\bf C}_{r',p'}$ for each $k\in\bb N$, and moreover, the constant $C$ in (\ref{eq.fk}) is independent of $k$. By straightforward applications of Fatou's Lemma, one may show that
\[
\n C_{r'}(F)\leq\liminf_{k\ra\infty}\n C_{r'}(F_k)\qquad\text{on }\partial\Omega,
\]
whence, by   Fatou's Lemma again,
\[
\Vert\n C_{r'}(F)\Vert_{p'}^{p'}\leq\liminf_{k\ra\infty}\Vert\n C_{r'}(F_k)\Vert_{p'}^{p'}\leq C^{p'},
\]
where $C$ is the constant from (\ref{eq.fk}). Thus $F\in{\bf C}_{r',p'}$, as desired.

{\bf Absolute continuity of $\mu$ with respect to $\sigma$ on  compact subsets of $\partial\Omega$.} From (\ref{eq.decompose1}) and the arguments above, we fully understand the behavior of $\mu$ on $\Omega$. Since $\overline\Omega=\Omega\cup\partial\Omega$ and the latter two sets are disjoint, it remains only to understand the behavior of $\mu$ on $\partial\Omega$.  First we show that $\mu|_{\partial\Omega}\ll\sigma$. Fix a compact set $K\subset\overline\Omega$, and a non-empty  Borel set $S\subset K\cap\partial\Omega$. Since $\mu$ is  regular, we have that
\[
\mu(S)=\inf\big\{\mu(U)~:~ S\subset U  \text{ and } U \text{ is open in }\overline\Omega\big\}.
\]
If $U$ is an arbitrary open set in $\overline\Omega$ containing $S$, then note that $S\subset U\cap\partial\Omega$, whence 
\[
\mu(S)\leq\mu(U\cap\partial\Omega)\leq\mu(U).
\]
It follows that
\[
\mu(S)=\inf\big\{\mu(U\cap\partial\Omega)~:~ S\subset U  \text{ and } U \text{ is open in }\overline\Omega\big\}.
\]
Now fix an arbitrary open set $U$ in $\overline\Omega$ containing $S$, and let $V:=U\cap\partial\Omega$. Note that $V$ is relatively open in $\partial\Omega$, although $V$ is of course not open in $\overline\Omega$. Let $z\in S$, $\ell:=\diam K$ and write
\[
V_1:= V\cap B(z,2\ell).
\]
It is clear that $S\subset V_1\subset V$ and that $V_1$ is relatively open in $\partial\Omega$. Now let $h:\partial\Omega\ra\bb R$ be defined by
\[
h(\xi):={\bf 1}_{V_1}(\xi),\qquad\xi\in\partial\Omega.
\]
Then it is straightforward that $0\leq h\leq1$, $h\in L^p(\partial\Omega,\sigma)$, $\supp h=V_1$, and $h$ is continuous on each point of $V_1$, since $V_1$ is relatively open in $\partial\Omega$. Now, in the notation of Lemma \ref{lm.approxid}, let 
\[
f_k:=P_{\frac1k}h,\qquad\text{ for each  }k\in\bb N.
\]
By Lemma \ref{lm.approxid}, we see that $\{f_k\}\subset C_c(\partial\Omega)$ (since $h$ has bounded support), each $f_k$ is non-negative on $\partial\Omega$,
\begin{equation}\label{eq.ptwise}
f_k\lra h,\qquad\text{pointwise on }V_1 \text{ as }k\ra\infty,
\end{equation}
and $f_k\ra h$ strongly in $L^p(\partial\Omega,\sigma)$ as $k\ra\infty$. Then note that
\[
\mu(V_1)=\int_{\partial\Omega}h\,d\mu=\int_{V_1}h\,d\mu\leq\liminf_{k\ra\infty}\int_{V_1}f_k\,d\mu\leq\liminf_{k\ra\infty}\int_{\partial\Omega}f_k\,d\mu,
\]
where  we used Fatou's Lemma\footnote{Note carefully that we strongly needed $f_k\ra h$ pointwise in (\ref{eq.ptwise}), and not merely pointwise $\sigma$-a.e., since we do not a priori assume any relation between $\mu$ and $\sigma$.}. Next, for each $k$, let $u_k$ be the Varopoulos extension of $f_k$ onto $\Omega$, as in Lemma \ref{lm.varopoulos}. Then
\begin{equation}\label{eq.var1}
\int_{\partial\Omega}f_k\,d\mu=\int_{\partial\Omega}u_k\,d\mu=\Big[T(u_k)-\int_\Omega u_kF\,dm\Big] \leq\Big[C+\Vert F\Vert_{{\bf C}_{r',p'}}\Big]\Vert\wt{\m N}_r(u_k)\Vert_{L^p(\partial\Omega,\sigma)},
\end{equation}
where we used (\ref{eq.decompose1}) and Carleson's Theorem. In the previous estimate, $C$ is independent of $k$, $K$, $S$, $V$, $z$, and $\ell$.
Using Lemma \ref{lm.varopoulos}, we have that
\[
\Vert\wt{\m N}_r(u_k)\Vert_{L^p(\partial\Omega,\sigma)}\lesssim\Vert f_k\Vert_{L^p(\partial\Omega,\sigma)},\qquad\text{for each }k\in\bb N,
\]
and since $\Vert f_k\Vert_p\ra\sigma(V_1)^{\frac1p}$ as $k\ra\infty$, from the previous estimates we deduce that
\[
\mu(V_1)\leq C\sigma(V_1)^{\frac1p},
\]
where $C$ is independent of $S$, $K$, $V$, $z$, and $\ell$. In particular,
\[
\mu(S)\leq C\sigma(V)^{\frac1p}.
\]
Since $V=U\cap\partial\Omega$ and $U$ was an arbitrary open set in $\overline\Omega$ containing $S$, taking infimum over all such $U$ in the previous estimate and using the outer regularity of $\sigma$, we conclude that
\[
\mu(S)\leq\sigma(S)^{\frac1p}.
\]
Hence $\mu(S)\ra0$ as $\sigma(S)\ra0$, so that $\mu\ll\sigma$ on compact sets contained in $\partial\Omega$. Thus again by a diagonalization argument we conclude that there exists a non-negative Borel measurable function $g:\partial\Omega\ra\bb R$ such that $d\mu|_S=g\,d\sigma|_S$ for any Borel set $S\subset\partial\Omega$. In particular, for any $u\in C_c(\overline\Omega)$,
\[
\int_{\partial\Omega}u\,d\mu=\int_{\partial\Omega}ug\,d\sigma.
\]
Combining this last identity with (\ref{eq.decompose1}) gives the desired identity (\ref{eq.omega}), but it remains only to show that  $g\in L^{p'}(\partial\Omega,\sigma)$.

{\bf Integrability property of $g$.} In fact this is already essentially contained in the previous argument: we know that $g\in L^{p'}(\partial\Omega,\sigma)$ if and only if there exists a constant $C>0$ so that for all $f\in C_c(\partial\Omega)$, the estimate
\[
\Big|\int_{\partial\Omega}fg\,d\sigma\Big|\leq C\Vert f\Vert_{L^p(\partial\Omega,\sigma)}
\]
holds. Given an arbitrary $f\in C_c(\partial\Omega)$, let $u$ be the Varopoulos extension of $f$ onto $\Omega$ (see Lemma \ref{lm.varopoulos}). Then, proceeding similarly as in (\ref{eq.var1}), we see that
\[
\Big|\int_{\partial\Omega}fg\,d\sigma\Big|\leq\Big[C+\Vert F\Vert_{{\bf C}_{r',p'}}\Big]\Vert\wt{\m N}_r(u)\Vert_{L^p(\partial\Omega,\sigma)}.
\]
The desired estimate follows by using the fact that
\[
\Vert\wt{\m N}_r(u)\Vert_{L^p(\partial\Omega,\sigma)}\lesssim\Vert f\Vert_{L^p(\partial\Omega,\sigma)}.
\]
Thus $g\in L^{p'}(\partial\Omega,\sigma)$. This concludes the proof of Claim \ref{claim.decomp}, and with that, the end of the proof of Theorem \ref{thm.dual}.\hfill{$\square$}

\begin{remark}[On the dual of $\hat{\bf N}_{r,p}$]\label{rm.different2} As mentioned in Remark \ref{rm.different}, the space ${\bf N}_{r,p}$ that we consider is a strict subset of the space $\hat{\bf N}_{r,p}$ considered originally in \cite{hr13,mpt25}. It is natural to wonder if there is an analogous characterization of the dual space of $\hat{\bf N}_{r,p}$ as the one achieved in Theorem \ref{thm.dual} for the space ${\bf N}_{r,p}$. In fact, there is a corresponding natural result, easily gathered from the proof of Theorem \ref{thm.dual}, as follows.
	
First, each bounded linear functional $T$ on $\hat{\bf N}_{r,p}$ must still obey identity (\ref{eq.ide}) (valid for $u\in C_c(\overline\Omega)$) with the measure $\mu$ still satisfying the same decomposition from Claim \ref{claim.decomp}. Nevertheless, the identity (\ref{eq.decomp}) does not necessarily hold for arbitrary $u\in\hat{\bf N}_{r,p}$. 

Second, if $g\in L^{p'}(\partial\Omega,\sigma)$, and we define the linear functional $\ell_g$ on $C(\overline{\Omega})$ by
\[
(\ell_g,u):=\int_{\partial\Omega}g(\xi)u(\xi)\,d\sigma(\xi),\qquad u\in C(\overline\Omega),
\]
then by the Hahn-Banach Theorem and the fact that
\[
|(\ell_g,u)|\leq\Vert g\Vert_{L^{p'}(\sigma)}\Vert u\Vert_{L^p(\sigma)}\leq\Vert g\Vert_{L^{p'}(\sigma)}\Vert u\Vert_{{\bf N}_{r,p}},
\]
we may extend $\ell_g$ to a bounded linear functional on $\hat{\bf N}_{r,p}$. This is true despite the fact that the generic functions in $\hat{\bf N}_{r,p}$ do not have well-defined traces on $\partial\Omega$. The price paid from using the Hahn-Banach extension is the loss of uniqueness of the extended functional.
\end{remark}


\section{Proof of the Approximation Theorem}\label{sec.approx}

In this section we prove Theorem \ref{thm.approx}. Before we do so, let us briefly remark on some of the hypotheses and choice of parameters in the theorem.

\begin{remark}\label{rm.details} The boundedness of the domain is not essential, but it affords us a relatively simpler  proof. The parameter $2$ in ${\bf N}_{2,p}$ and in ${\bf C}_{2,p}$   can be relaxed to $r\in[2,\infty)$ everywhere in Theorem \ref{thm.approx}. Furthermore, if $A$ is locally Lipschitz continuous in $\Omega$, then the parameter can be relaxed to $r\in(1,\infty)$. The sequence $\{F_j\}_j$ depends on this parameter, but the constant in (\ref{eq.uniformc}) does not.
\end{remark}

\noindent\emph{Proof of Theorem \ref{thm.approx}.}

{\bf Part 1. Case $g$ is Lipschitz continuous.} 

\emph{Proof of \ref{item.c1}.} Suppose that $g\in\Lip(\partial\Omega)$ and let $\Phi=\Phi_g$ be the Varopoulos  extension of $g$ onto $\Omega$ as in Lemma \ref{lm.varopoulos}. Now, for each $j\in\bb N$ large enough (namely, such that $2^{-j}<\diam\partial\Omega/2$), let $\eta_j$ be a cut-off function satisfying that $\eta_j\equiv1$ on $B(\partial\Omega,2^{-(j+1)})$, $\supp\eta_j\subset B(\partial\Omega,2^{-j})$, $|\nabla\eta_j|\lesssim2^{j}$, and $0\leq\eta_j\leq1$. We set 
\begin{equation}\label{eq.fj0}
\tilde F_j:=-A\nabla(\Phi\eta_j).
\end{equation}
It is clear that 
\begin{equation}\label{eq.a1}
	|\tilde F_j|\leq|\eta_jA\nabla\Phi|+|\Phi A\nabla\eta_j|,
\end{equation}
and by Lemma \ref{lm.varopoulos} \ref{item.v2} and the boundedness of $A$, we immediately have that 
\[
\Vert\n C_\infty((\eta_jA\nabla\Phi))\Vert_{L^p(\sigma)}\lesssim\Vert g\Vert_{L^p(\sigma)}.
\]
It remains to control the second term on the right-hand side of (\ref{eq.a1}). We claim that
\begin{equation}\label{eq.carlesonineq}
	\n C_\infty(\Phi A\nabla\eta_j)\lesssim\m M(\m N\Phi),\qquad\text{pointwise on }\partial\Omega.
\end{equation}
Taking the claim for granted for a moment, we easily see by the properties of the Hardy-Littlewood maximal function and Lemma \ref{lm.varopoulos} \ref{item.v2} that
\[
\Vert\n C_\infty(\Phi A\nabla\eta_j)\Vert_{L^p(\sigma)}\lesssim\Vert\m N(\Phi)\Vert_{L^p(\sigma)}\lesssim\Vert g\Vert_{L^p(\sigma)}.
\]
These bounds together imply that $\tilde F_j\in{\bf C}_{\infty,p}$, and by H\"older's inequality in fact we have that $\tilde F_j\in{\bf C}_{q,p}$ for all $q\in(1,\infty]$. Now, for each $q\in(1,\infty)$, by the density of $\Lip_c(\Omega)$ in ${\bf C}_{q,p}$ (see \cite[Lemma 2.5]{mpt25}), we may find compactly supported Lipschitz vector functions $F_j$ such that
\begin{equation}\label{eq.fj}
\Vert\n C_q(F_j-\tilde F_j)\Vert_{L^p(\sigma)}\leq2^{-j}\Vert g\Vert_{L^p(\sigma)}.
\end{equation}
This finishes the proof of \ref{item.c1} when $g$ is Lipschitz continuous, modulo the claim.

\emph{Proof of Claim (\ref{eq.carlesonineq}).} Fix  $\xi\in\partial\Omega$. Note that $\nabla\eta_j$ is supported on the strip $S_j:=B(\partial\Omega,2^{-j})\setminus B(\partial\Omega,2^{-(j+1)})$, which allows us to write 
\begin{equation}\label{eq.c2}
	\n C_\infty(\Phi A\nabla\eta_j)(\xi)\lesssim\sup_{r>2^{-(j+1)}}\frac{2^j}{r^n}\int_{\Omega\cap B(\xi,r)\cap S_j}~\sup_{y\in B(x,\delta(x)/2)}|\Phi(y)|\,dm(x). 
\end{equation} 
Now fix $r\in(2^{-(j+1)},\diam\partial\Omega/2)$ and  $M\in\bb Z$ such that $2^M\geq r>2^{M-1}$. We need to decompose $\Omega\cap B(\xi,r)\cap S_j$ into annular regions. To this end, write
\[
A_{j,k}:=\big\{x\in\Omega\cap S_j~:~x\in B(\xi,2^{k+1})\setminus B(\xi,2^k)\big\},\qquad k\in\bb Z, ~k\geq-j+3.
\]
Next, note that for each $k\geq-j+3$, the collection of balls $\{B(x,\frac{\delta(x)}4)\}_{x\in A_{j,k}}$ is a Besicovitch covering for $A_{j,k}$, and so by the Besicovitch Covering Theorem we extract a finite collection $\{V_\ell\}_\ell=\{B(x_\ell,\frac{\delta(x_\ell)}4)\}_\ell$ of such balls with uniformly finite overlap (depending only on dimension). Hence
\begin{multline}\label{eq.annular}
	\int_{\Omega\cap B(\xi,r)\cap S_j}~\sup_{y\in B(x,\delta(x)/2)}|\Phi(y)|\,dm(x)\\ \leq\Big[\int_{\Omega\cap S_j\cap B(\xi,2^{-j+3})}~\sup_{y\in B(x,\delta(x)/2)}|\Phi(y)|\,dm(x)~+~\sum_{k=-j+3}^{M-1}\sum_\ell\int_{V_\ell}~\sup_{y\in B(x,\delta(x)/2)}|\Phi(y)|\,dm(x)\Big]\\=:I_{00}+\sum_{k=-j+3}^{M-1}I_k.
\end{multline}

We control $I_{00}$ first. Note that 
\[
\Omega\cap S_j\cap B(\xi,2^{-j+3})\subset\gamma_\alpha(\zeta),\qquad\text{for all }\zeta\in\partial\Omega\cap B(\xi,2^{-j+3})
\]
for $\alpha$ large enough (namely $\alpha\geq31$), and so 
\[
\sup_{y\in B(x,\delta(x)/2)}|\Phi(y)|\leq\m N_\alpha(\Phi)(\zeta),\qquad\text{for all }x\in\Omega\cap S_j\cap B(\xi,2^{-j+3}) \text{ and all }\zeta\in\partial\Omega\cap B(\xi,2^{-j+3}).
\]
It follows that
\[
I_{00}\leq\int_{\Omega\cap S_j\cap B(\xi,2^{-j+3})}\dashint_{\partial\Omega\cap B(\xi,2^{-j+3})}\m N_\alpha(\Phi)(\zeta)\,d\sigma(\zeta)\,dm(x)\approx2^{-j}\int_{\partial\Omega\cap B(\xi,2^{-j+3})}\m N_\alpha(\Phi)(\zeta)\,d\sigma(\zeta).
\]

We turn to $I_{k}$ for fixed $k\geq-j+3$. It is clear that for each $\ell$,
\[
\big\{y\in\Omega~:~y\in B(x,\tfrac{\delta(x)}2) \text{ for some }x\in V_\ell\big\}\subseteq B(x_\ell,\tfrac78\delta(x_\ell)),
\]
and if $\zeta_\ell\in\partial\Omega$ is such that $|x_\ell-\zeta_\ell|=\dist(x_\ell,\partial\Omega)=\delta(x_\ell)$, then
\[
B(x_\ell,\tfrac78\delta(x_\ell))\subset\gamma_\alpha(\zeta)\qquad\text{for each }\zeta\in W_\ell:=\partial\Omega\cap B(\zeta_\ell,\tfrac{1}8\delta(x_\ell)).
\]
for $\alpha\geq\frac{11}4$. Therefore, for each $\ell$,
\[
\sup_{y\in B(x,\delta(x)/2)}|\Phi(y)|\leq\m N(\Phi)(\zeta),\qquad\text{for all }\zeta\in W_\ell \text{ and all }x\in V_\ell.
\]
Hence, using that $\delta(x_\ell)\approx 2^{-j}$, it follows that
\begin{multline}\label{eq.c1}
	\sum_\ell\int_{V_\ell}~\sup_{y\in B(x,\delta(x)/2)}|\Phi(y)|\,dm(x)\leq\sum_\ell\int_{V_\ell}\dashint_{W_\ell}\m N(\Phi)(\zeta)\,d\sigma(\zeta)\,dm(x)\\ \lesssim\sum_\ell2^{-j(n+1)}\dashint_{W_\ell}\m N(\Phi)(\zeta)\,d\sigma(\zeta)\lesssim2^{-j}\sum_\ell\int_{W_\ell}\m N(\Phi)(\zeta)\,d\sigma(\zeta).
\end{multline}

We now show that there exists $N\in\bb N$ (depending only on dimension) such that for each $k$, each $\zeta\in\cup_\ell W_\ell$ can belong to at most $N$   sets in the union. Indeed, if $W_\ell$ and $W_m$ intersect for some $\ell,m$, then from the definition of $W_\ell$ and repeated application of the triangle inequality we deduce that $|x_\ell-x_m|<\frac942^{-j}$. It follows that if $\zeta\in\cap_{\ell'}W_{\ell'}$ for a subcollection  $\{W_{\ell'}\}_{\ell'}$, then
\[
\cup_{\ell'}V_{\ell'}\subset B(x_m,\tfrac522^{-j}),\qquad\text{for some }x_m \text{ in the subcollection }\{x_{\ell'}\}_{\ell'}.
\]
As the $V_{\ell}$ have uniformly finite overlap, it follows from the last containment that there can be at most a uniformly finite number $N$ of elements in the subcollection $\{V_{\ell'}\}_{\ell'}$. Using this result in (\ref{eq.c1}), we deduce that 
\[
\sum_\ell\int_{V_\ell}~\sup_{y\in B(x,\delta(x)/2)}|\Phi(y)|\,dm(x)\lesssim2^{-j}\int_{\cup_\ell W_\ell}\m N(\Phi)(\zeta)\,d\sigma(\zeta).
\]
Note that if $\zeta\in\cup_\ell W_\ell$, then $|\zeta-\xi|\approx2^k$ (with universal constants independent of dimension). Indeed, this follows by straightforward applications of the triangle inequality. Hence, there exist universal $a,b>0$ such that if $C_k:=\{\zeta\in\partial\Omega~:~a2^k\leq|\zeta-\xi|\leq b2^{k+1}\}$, then $\cup_\ell W_\ell\subset C_k$, and so
\[
I_k\lesssim2^{-j}\int_{C_k}\m N(\Phi)(\zeta)\,d\sigma(\zeta),\qquad k\geq-j+3.
\]
It is easy to see that the $C_k$'s have uniformly finite overlap, and so
\[
\sum_{k=-j+3}^{M-1}I_k\lesssim2^{-j}\int_{\partial\Omega\cap B(\xi,b2^M)\setminus B(\xi,a2^{-j+3})}\m N(\Phi)(\zeta)\,d\sigma(\zeta).
\]
Adding our estimates for $I_{00}$ and the sum of the $I_k$'s, from (\ref{eq.annular}) we see that
\[
\int_{\Omega\cap B(\xi,r)\cap S_j}~\sup_{y\in B(x,\delta(x)/2)}|\Phi(y)|\,dm(x)\lesssim2^{-j}\int_{\partial\Omega\cap B(\xi,b2^M)}\m N(\Phi)(\zeta)\,d\sigma(\zeta).
\]
Using this estimate in (\ref{eq.c2}), we have that
\begin{equation*} 
	\n C_\infty(\Phi A\nabla\eta_j)(\xi)\lesssim\sup_{r>2^{-(j+1)}}\frac{1}{r^n}\int_{\partial\Omega\cap B(\xi,Cr)}\m N(\Phi)(\zeta)\,d\sigma(\zeta)\lesssim\m M(\m N(\Phi))(\xi),
\end{equation*}
as desired.

\emph{Proof of \ref{item.c2}.} We continue to assume that  $g\in\Lip(\partial\Omega)$. Note that, in this case, by the Wiener regularity theorem we have that the solution to (\ref{eq.dirichlet0}) satisfies $u\in C(\overline\Omega)\cap W^{1,2}(\Omega)$ and $u|_{\partial\Omega}=g$ (both in the pointwise sense and in the sense of traces in the space $\m T(\partial\Omega)=W^{1,2}(\Omega)/W_0^{1,2}(\Omega)$). Let $w_j$ be as in \ref{item.c2} (with $F_j$ defined in the paragraph leading up to (\ref{eq.fj}), with $q=2$), and let us begin with \ref{item.c2} \ref{item.c2c}. By hypothesis $(\operatorname{D}_p^L)$ is solvable in $\Omega$, which implies by \cite[Theorem 1.2]{mpt25} that $(\operatorname{PD}_p^L)$ is solvable in $\Omega$, and therefore
\[
\Vert\widetilde{\m N}_2(w_j)\Vert_{L^p(\sigma)}\lesssim\Vert\n C_2(F_j)\Vert_{L^p(\sigma)}\lesssim\Vert g\Vert_{L^p(\sigma)},
\]
where we used (\ref{eq.uniformc}). Note also that  $u\in{\bf N}_{2,p}$ trivially. Let $p'$ be the H\"older conjugate of $p$. By \cite{hr13,mpt25}, we know that $({\bf C}_{2,p'})^*={\bf N}_{2,p}$; thus, to show the weak-$*$ convergence in ${\bf N}_{2,p}$ is precisely to show that
\begin{equation}\label{eq.conv}
\int_\Omega(w_j-u)H\,dm\lra0,\quad\text{as }j\ra\infty, \text{ for each }H\in{\bf C}_{2,p'}.
\end{equation} 

Assume first that $H\in\Lip_c(\Omega)$, and let $v\in W_0^{1,2}(\Omega)$ be the weak solution to the Poisson problem
\begin{equation}\label{eq.poissonregularity}
	\left\{\begin{aligned}-\dv A^T\nabla v&=H,\qquad&\text{in }&\Omega,\\ v&=0,\qquad&\text{on }&\partial\Omega.\end{aligned}\right. 
\end{equation}
Using the conormal derivative (see Definition \ref{def.conormal}), it follows that
\begin{equation}\label{eq.conv2}
	\int_\Omega(w_j-u)H\,dm=\int_\Omega A^T\nabla v\nabla(w_j-u)\,dm-(\partial_{\nu_{A^T}}v,w_j-u)=\int_\Omega A^T\nabla v\nabla(w_j-u)\,dm+(\partial_{\nu_{A^T}}v,u),
\end{equation}
where we used that $(\partial_{\nu_{A^T}}v,w_j)\equiv0$ for each $j$.
On the other hand, note that
\begin{multline}\label{eq.conv3}
	\int_\Omega A^T\nabla v\nabla(w_j-u)\,dm=\int_\Omega A^T\nabla v\nabla w_j\,dm-\int_\Omega A^T\nabla v\nabla u\,dm\\ =\int_\Omega A\nabla w_j\nabla v\,dm-\int_\Omega A\nabla u\nabla v\,dm=\int_\Omega F_j\nabla v\,dm,
\end{multline}
since $Lu=0$, $Lw_j=-\dv F_j$ in $\Omega$ and $v\in W_0^{1,2}(\Omega)$. Next, using the definition of $F_j$ and $\tilde F_j$ (see (\ref{eq.fj0}) and (\ref{eq.fj})), we have that
\begin{equation}\label{eq.conv31}
\int_\Omega F_j\nabla v\,dm=\int_\Omega(F_j-\tilde F_j)\nabla v\,dm-\int_\Omega A\nabla(\Phi\eta_j)\nabla v\,dm,
\end{equation}
and
\begin{equation}\label{eq.fjapprox}
\Big|\int_\Omega(F_j-\tilde F_j)\nabla v\,dm\Big|\lesssim\Vert\n C_2(F_j-\tilde F_j)\Vert_{L^p(\sigma)}\Vert\widetilde{\m N}_2(\nabla v)\Vert_{L^{p'}(\sigma)}\lesssim2^{-j}\Vert g\Vert_{L^p(\sigma)}\Vert\n C_2(H)\Vert_{L^p(\sigma)},
\end{equation}
where we used that the Poisson-Regularity problem $(\operatorname{PR}_{p'}^{L^*})$ is solvable in $\Omega$, since $(\operatorname{D}_p^L)$ is solvable in $\Omega$ (see \cite[Definition 1.6 and Theorem 1.8]{mpt25}). Next, we see that
\begin{equation}\label{eq.conv32}
-\int_\Omega A\nabla(\Phi\eta_j)\nabla v\,dm=-\int_\Omega A^T\nabla v\nabla(\Phi\eta_j)\,dm=-(\partial_{\nu_{A^T}}v,\Phi\eta_j)-\int_\Omega H\Phi\eta_j\,dm
\end{equation}
where we used the definition of the conormal derivative. Putting (\ref{eq.conv2}) together with  (\ref{eq.conv3}), (\ref{eq.conv31}), and (\ref{eq.conv32}), we get that
\begin{equation}\label{eq.conv4}
	\int_\Omega(w_j-u)H\,dm=-\int_\Omega H\Phi\eta_j\,dm~+~(\partial_{\nu_{A^T}}v,u-\Phi\eta_j)~+~\int_\Omega(F_j-\tilde F_j)\nabla v\,dm.
\end{equation}
The third term on the right-hand side of (\ref{eq.conv4}) vanishes as $j\ra\infty$ due to (\ref{eq.fjapprox}). The second term is identically $0$ for all $j$ since $u-\Phi\eta_j\in W_0^{1,2}(\Omega)$, and this is true by virtue of the fact that both $u$ and $\Phi\eta_j$ are in $W^{1,2}(\Omega)\cap C(\overline\Omega)$ and $u|_{\partial\Omega}=\Phi\eta_j|_{\partial\Omega}=g$ (see \cite{sz99}). As for the first term, notice that, since $H$ has compact support in $\Omega$ and $\supp\eta_j\subset B(\partial\Omega,2^{-j})$, it follows that for all $j$ large enough (depending on the support of $H$), this term becomes identically $0$. This establishes the convergence (\ref{eq.conv}) in the case that $H\in\Lip_c(\Omega)$. 

Now suppose that $H\in{\bf C}_{2,p'}$ is arbitrary. Since compactly supported Lipschitz functions are dense in ${\bf C}_{2,p'}$ (see \cite[Lemma 2.5]{mpt25}), there is a sequence $\{H_k\}_{k\in\bb N}\subset\Lip_c(\Omega)$ such that 
\[
\Vert\n C_2(H-H_k)\Vert_{L^{p'}(\sigma)}\ra0\qquad\text{as }k\ra\infty.
\]
We may write for each $j,k\in\bb N$,
\begin{equation}\label{eq.dense}
\int_\Omega(w_j-u)H\,dm=\int_\Omega(w_j-u)(H-H_k)\,dm~+~\int_\Omega(w_j-u)H_k\,dm=:I_1+I_2
\end{equation}
Note that
\begin{multline*}
|I_1|\lesssim\Vert\widetilde{\m N}_2(w_j-u)\Vert_{L^p(\sigma)}\Vert\n C_2(H-H_k)\Vert_{L^{p'}(\sigma)}\\ \leq\Big[\Vert\widetilde{\m N}_2(w_j)\Vert_{L^p(\sigma)}+\Vert\widetilde{\m N}_2(u)\Vert_{L^p(\sigma)}\Big]\Vert\n C_2(H-H_k)\Vert_{L^{p'}(\sigma)}\\ \lesssim\Big[\Vert\n C_2(F_j)\Vert_{L^p(\sigma)}+\Vert g\Vert_{L^p(\sigma)}\Big]\Vert\n C_2(H-H_k)\Vert_{L^{p'}(\sigma)}\lesssim\Vert g\Vert_{L^p(\sigma)}\Vert\n C_2(H-H_k)\Vert_{L^{p'}(\sigma)},
\end{multline*}
where we used Proposition \ref{prop.duality} and the solvability of $(\operatorname{PD}_p^L)$ \cite[Theorem 1.2]{mpt25}. Fix $\ep>0$, and then by the previous calculations we see that we can fix $k$ large enough so that $|I_1|<\ep/2$ for all $j\in\bb N$. With $k$ fixed, since $H_k\in\Lip_c(\Omega)$, we have that for all $j$ large enough, $|I_2|<\ep/2$, as we already proved (\ref{eq.conv}) when $H\in\Lip_c(\Omega)$. We conclude that $w_j\ra u$ in the weak$-*$ topology in ${\bf N}_{2,p}$.

That this convergence cannot be improved to the weak topology in ${\bf N}_{2,p}$ unless $g\equiv0$ follows from the fact that $L^{p'}(\partial\Omega)\subset({\bf N}_{2,p})^*$ from Theorem \ref{thm.dual}. More precisely, suppose that $w_j\ra u$ weakly in ${\bf N}_{2,p}$; then $\langle w_j,\phi\rangle\ra\langle u,\phi\rangle$ for each $\phi\in({\bf N}_{2,p})^*$. In particular, by Theorem \ref{thm.dual} this is true for arbitrary $\phi\in L^{p'}(\partial\Omega)$. However, since $w_j\equiv0$ on $\partial\Omega$ for each $j$, it follows that $\langle w_j,\phi\rangle\equiv0$ for any $\phi\in L^{p'}(\partial\Omega)$, which would imply that $\langle u,\phi\rangle=0$ for any $\phi\in L^{p'}(\partial\Omega)$, hence $g=u|_{\partial\Omega}\equiv0$ on $\partial\Omega$. This finishes the proof of \ref{item.c2} \ref{item.c2c}. 

We now prove \ref{item.c2} \ref{item.c2a} and \ref{item.c2} \ref{item.c2b}.  Recall that $\tilde F_j$ was defined in (\ref{eq.fj0}). Let $\tilde w_j\in W_0^{1,2}(\Omega)$ be the unique weak solution to the Poisson problem
\begin{equation}\label{eq.poissondirichlet1} \nonumber
	\left\{\begin{aligned}-\dv A\nabla\tilde w_j&=-\dv\tilde F_j,\qquad&\text{in }&\Omega,\\ \tilde w_j&=0,\qquad&\text{on }&\partial\Omega.\end{aligned}\right. 
\end{equation}
Then, for each $j,k\in\bb N$, note that $\tilde w_j-\tilde w_k$ must be the unique weak solution in $W_0^{1,2}(\Omega)$ to the problem
\begin{equation}\label{eq.poissondirichlet2} 
	\left\{\begin{aligned}-\dv A\nabla(\tilde w_j-\tilde w_k)&=-\dv(\tilde F_j-\tilde F_k),\qquad&\text{in }&\Omega,\\ \tilde w_j-\tilde w_k&=0,\qquad&\text{on }&\partial\Omega.\end{aligned}\right. 
\end{equation} 
On the other hand, note that the function $\Phi\eta_j-\Phi\eta_k\in W_0^{1,2}(\Omega)$ is also a solution to (\ref{eq.poissondirichlet2}). By the uniqueness, it follows that
\begin{equation}\label{eq.wdiff}
\tilde w_j-\tilde w_k\equiv\Phi(\eta_j-\eta_k),\qquad\text{for each }j,k\in\bb N.
\end{equation}
Let us now fix a compact set $K\subset\Omega$ and use the above observation to show that $\{w_j\}_j$ (defined by (\ref{eq.poissondirichlet0})) is a Cauchy sequence in $L^2(K)$. Indeed,
\[
w_j-w_k=(w_j-\tilde w_j)+(\tilde w_j-\tilde w_k)+(\tilde w_k-w_k)=T_1+T_2+T_3.
\]
Since $K$ is compact, there exists $\ell\in\bb N$ (depending on $K$) so that $2^{-\ell+1}<\dist(K,\partial\Omega)$. It follows that for all $j,k\geq\ell$, $\tilde w_j-\tilde w_k\equiv0$ on $K$, by (\ref{eq.wdiff}). Hence $T_2\ra0$ as $j,k\ra\infty$. As for $T_1$, note that
\begin{equation}\label{eq.closew}
\Vert w_j-\tilde w_j\Vert_{L^2(K)}\lesssim\Vert\widetilde{\m N}_2(w_j-\tilde w_j)\Vert_{L^p(\sigma)}\lesssim\Vert\n C_2(F_j-\tilde F_j)\Vert_{L^p(\sigma)}\leq2^{-j}\Vert g\Vert_{L^p(\sigma)},
\end{equation}
where we used the solvability of $(\operatorname{PD}_p^L)$ \cite[Theorem 1.2]{mpt25} and (\ref{eq.fjapprox}). Hence $T_1$ vanishes in $L^2(K)$ norm as $j\ra\infty$. In a completely analogous manner, the term $T_3$ will vanish in the $L^2(K)$ norm as $k\ra\infty$. 

We have thus shown that $\{w_j\}_j$ is Cauchy in $L^2(K)$, and therefore there exists   $w\in L^2(K)$ such that $w_j\ra w$ strongly in $L^2(K)$. By considering an exhausting sequence of compact sets $K_n$ in $\Omega$ (i.e. $K_m\subset K_{m+1}\subset\Omega$ and $\dist(K_m,\partial\Omega)\ra0$) and a straightforward diagonalization argument, we may take $w$ to be defined in all of $\Omega$ and $w_j\ra w$ in $L^2(K)$ for all compact sets $K\subset\Omega$.

By the previous result and the Caccioppoli inequality, it is an exercise to see that for each compact $K\subset\Omega$, $\{\nabla w_j\}$ is Cauchy in $L^2(K)$, and hence $w$ is weakly differentiable, and $w_j\ra w$ strongly in $W^{1,2}(K)$. In fact, as we have already shown \ref{item.c2} \ref{item.c2c}, from (\ref{eq.conv}) and the fact that $L^2(K)\subset{\bf C}_{2,p'}$, it follows that $\{w_j\}_j$ converges to $u$ weakly in $L^2(K)$, and so we must have that $w\equiv u$ as elements in $L^2(K)$, and hence pointwise almost everywhere in $\Omega$.  At this juncture we may conclude that $w_j\ra u$ strongly in $W^{1,2}_{\loc}(\Omega)$, which finishes the proof of \ref{item.c2} \ref{item.c2b}.

Finally, by the DeGiorgi-Nash-Moser theory, for each compact $K\subset\Omega$, the $C^\alpha(K)$ norm of $w_j-u$ is controlled by $\Vert w_j-u\Vert_{L^2(\tilde K)}$, where $\tilde K$ is compact and $K\subset\tilde K\subset\Omega$ (see for instance \cite[Theorem 4.11]{hlin}. As $w_j\ra u$ strongly in $L^2_{\loc}(\Omega)$, it follows that $w_j-u$ converges to $0$ in the $C^\alpha(K)$ norm. This ends the proof of \ref{item.c2} \ref{item.c2a}.

{\bf Part 2. Case $g\in L^p(\partial\Omega,\sigma)$.} 

Note that we have fully proven Theorem \ref{thm.approx} in the case that $g\in\Lip(\partial\Omega)$, and we will use this in the sequel. Fix $g\in L^p(\partial\Omega,\sigma)$. Since $\Lip(\partial\Omega,\sigma)$ is dense in $L^p(\partial\Omega,\sigma)$, we may find a sequence $\{g_k\}_k\subset\Lip(\partial\Omega,\sigma)$ such that $g_k\ra g$ strongly in $L^p(\partial\Omega,\sigma)$. For each $k\in\bb N$, let $u_k\in W^{1,2}(\Omega)\cap C(\overline\Omega)$ be the unique solution to the Dirichlet problem 
\begin{equation}\label{eq.dirichlet1}\nonumber
	\left\{\begin{aligned}-\dv A\nabla u_k&=0,\qquad&\text{in }&\Omega,\\ u_k&=g_k,\qquad&\text{on }&\partial\Omega,\end{aligned}\right.
\end{equation}
and let $\{F_{k,j}\}_j$ and $\{w_{k,j}\}_j$ be the sequences described in the conclusion of Theorem \ref{thm.approx} associated to $g_k$, and constructed explicitly as described in the proof of Part 1 above.

Now fix a  strictly increasing subsequence of natural numbers $\{j_k\}_k$ and define
\[
F_k:=F_{k,j_k},\qquad w_k:=w_{k,j_k}.
\]

\emph{Proof of \ref{item.c1}.} Simply note that for each $k\in\bb N$,
\[
\Vert\n C_2(F_{k,j_k})\Vert_{L^p(\sigma)}\lesssim\Vert g_k\Vert_{L^p(\sigma)}\lesssim\Vert g\Vert_{L^p(\sigma)},
\]
where we have used that the analogous result deduced in the case of Part 1.

\emph{Proof of \ref{item.c2}.} The statements in \ref{item.c2} \ref{item.c2a} and \ref{item.c2} \ref{item.c2b} follow directly from the analogous results of Part 1 and the straightforward observation that the convergence in those topologies also hold for $u_k\ra u$, since $u_k\ra u$ in ${\bf N}_{2,p'}$ (with essentially the same proof). We leave the details to the reader. 

We  show \ref{item.c2} \ref{item.c2c}. We must show that (\ref{eq.conv}) holds; as before, let us first fix $H\in\Lip_c(\Omega)$. For each $j\geq j_k$, we may write 
\begin{equation}\label{eq.poisson3}\nonumber
	\int_\Omega(w_k-u)H\,dm=\int_\Omega(w_{k,j_k}-w_{k,j})H\,dm+\int_\Omega(w_{k,j}-u_k)H\,dm+\int_\Omega(u_k-u)H\,dm=:T_1+T_2+T_3.
\end{equation}
Let us see why each of these terms vanish as $k\ra\infty$, starting with $T_3$. Note that
\[
|T_3|\lesssim\Vert\widetilde{\m N}_2(u_k-u)\Vert_{L^p(\sigma)}\Vert\n C_2(H)\Vert_{L^2(\sigma)}\lesssim\Vert g_k-g\Vert_{L^p(\sigma)}\Vert\n C_2(H)\Vert_{L^2(\sigma)},
\]
where we used that $(\operatorname{D}_p^L)$ is solvable in $\Omega$.  Next, set 
\[
\tilde F_{k,j}:=-A\nabla(\Phi_k\eta_j),\qquad k,j\in\bb N.
\]
and $\tilde w_{k,j}$ as the solution to the Poisson problem (\ref{eq.poissondirichlet0}) with $\tilde F_{k,j}$ in place of $F_j$. Then we write $T_1$ as follows:
\begin{equation}\label{eq.t1}
	T_1=\int_\Omega(w_{k,j_k}-\tilde w_{k,j_k})H\,dm+\int_\Omega(\tilde w_{k,j_k}-\tilde w_{k,j})H\,dm+\int_\Omega(\tilde w_{k,j}-w_{k,j})H\,dm=:T_{11}+T_{12}+T_{13}.
\end{equation}
Note that
\[
|T_{11}|\lesssim\Vert\widetilde{\m N}_2(w_{k,j_k}-\tilde w_{k,j_k})\Vert_{L^p(\sigma)}\Vert\n C_2(H)\Vert_{L^p(\sigma)}\lesssim2^{-j_k}\Vert g_k\Vert_{L^p(\sigma)}\leq2^{-j_k}\Vert g\Vert_{L^p(\sigma)},
\]
where we used the estimates in (\ref{eq.closew}). Similarly we may obtain 
\[
|T_{13}|\lesssim2^{-j}\Vert g\Vert_{L^p(\sigma)}\leq2^{-j_k}\Vert g\Vert_{L^p(\sigma)}.
\]
For $T_{12}$, we use  (\ref{eq.wdiff}) to see that
\[
T_{12}=\int_\Omega\Phi_k(\eta_{j_k}-\eta_j)H\,dm.
\]
By the definition of the cut-off functions $\eta_j$ and the fact that $H$ is compactly supported in $\Omega$, there exists $k^*$ large enough (depending on $H$) so that for all $k\geq k^*$, $\supp\eta_{j_k}\cap\supp H=\varnothing$. Since $\supp\eta_j\subset\supp\eta_{j_k}$ whenever $j\geq j_k$, we conclude that $T_{12}\equiv0$ for all $k\geq k^*$ and all $j\geq j_k$. These past few observations imply that $|T_1|$ can be made aribtrarily small for all $k$ large enough and all $j\geq j_k$. Finally,  $T_2$ vanishes as $j\ra\infty$ by (\ref{eq.conv}) for the case of Part 1, which we already showed.

The proof of (\ref{eq.conv})  in the full generality with  $H\in{\bf C}_{2,p'}$ follows similarly as in Part 1; one must use the following adaptation of (\ref{eq.carlesonineq}):
\begin{equation}\label{eq.carlesonineq2}
	\n C_\infty(\Phi_k A\nabla\eta_j)\lesssim\m M(\m N\Phi_k),\qquad\text{pointwise on }\partial\Omega,\qquad\text{for all }j,k\in\bb N.
\end{equation}
We leave the details to the interested reader. This ends the proof of \ref{item.c2} \ref{item.c2c}.\hfill{$\square$}

\bibliographystyle{alpha-sort-max} 
\bibliography{refs} 

\end{document}